\documentclass[final,leqno,onefignum,onetabnum]{siamltex1213}
\usepackage{amsmath}
\usepackage{amssymb}
\usepackage{graphicx}
\usepackage{color}
\usepackage{amsfonts}
\usepackage{amscd}
\usepackage{mathrsfs}
\usepackage{bm}
\usepackage{enumerate}
\usepackage{algorithmic, algorithm}

\renewcommand{\theequation}{\arabic{equation}}
\newcommand{\D}{\displaystyle}

\newcommand{\bx}{\bm{x}}
\newcommand{\be}{\bm{e}}
\newcommand{\bn}{\mathbf{n}}
\newcommand{\by}{\bm{y}}

\newcommand{\p}{\partial}

\newcommand{\bff}{{\bm{f}}}

\newcommand{\bfu}{{\bm{ u}}}

\newcommand{\bz}{\mathbf{z}}

\newcommand{\M}{{\Omega}}
\newcommand{\V}{{\mathcal V}}

\newcommand{\bV}{\mathbf{V}}

\newcommand{\mathd}{\mathrm{d}}

\newtheorem{remark}{\textbf{Remark}}[section]
\newtheorem{assumption}{\textbf{Assumption}}%[section]

\newcommand{\R}{\mathbb{R}}

\numberwithin{equation}{section}

\makeatother

\begin{document}

\title{A nonlocal Stokes system with volume constraints \thanks{The research of ZS was supported in part by grant NSFC 12071244. The research of QD was supported in part by NSF DMS-2012562 and  DMS-1937254.}}

\author{
Qiang Du
\thanks{Department of Applied Physics and Applied Mathematics, and Data Science Institute, Columbia University, New York, NY, 10027, USA,
 \textit{Email: qd2125@columbia.edu}
}
\and
Zuoqiang Shi
\thanks{Yau Mathematical Sciences Center, Tsinghua University, Beijing, 100084, China. 
Yanqi Lake Beijing Institute of Mathematical Sciences and Applications, Beijing, 101408, China. 	
	\textit{Email: zqshi@tsinghua.edu.cn.}}
}

\maketitle

\begin{abstract}
In this paper, we introduce a nonlocal model for linear steady Stokes system with physical no-slip boundary condition. We use the idea of volume constraint to 
enforce the no-slip boundary condition and prove that the nonlocal model is well-posed. We also show that and the solution of the nonlocal system converges to the solution of the original Stokes system as the nonlocality vanishes. 
\end{abstract}

\begin{keywords}
Nonlocal Stokes system · Nonlocal operators · Smoothed particle hydrodynamics · Incompressible flows · Well-posedness  · Local limit
\end{keywords}

\begin{AMS}
45P05\ , \  45A05\ ,  \ 35A23\ , \ 46E35
\end{AMS}

\section{Introduction}

Recently, nonlocal models and corresponding numerical methods have attracted much attention due to many successful applications. For example, in solid
mechanics, the theory of peridynamics \cite{Sill00} has been used as a possible alternative to conventional models of elasticity and fracture mechanics. 
Many numerical methods have also been developed to simulate nonlocal models like peridynamics based on 
rigorous mathematical analysis \cite{Du-SIAM,MD14,MD16,TD14,DLZ13,DJTZ13,ZD10}. 
Nonlocal methods are also successfully applied in image processing and data analysis 
\cite{Peyre09,LDMM,belkin2003led,Coifman05geometricdiffusions,LZ17,KLO17,RWP06,Gu04,LWYGL14,CLL15,MCL16,Lui11}. The idea of integral approximation is also applied to derive numerical scheme for solving PDEs on point cloud \cite{LSS,li2019point}. 
 
In this paper,
we study the nonlocal analog of the Stokes system in fluid
mechanics. Previously, nonlocal Stokes models have been proposed in \cite{DT17} and
\cite{hwi19} and analyzed subject to periodic boundary condition. 
In this paper, we consider the case of a nonlocal no-slip boundary  condition. 
More precisely, for
the conventional, local linear Stokes system on 
 a domain $\Omega\subset \mathbb{R}^n$, 
\begin{align}
\label{eq:stokes}
\left\{\begin{array}{rcl}
  \Delta \bfu(\bx)-\nabla p(\bx) &= & \bff(\bx),\quad \bx\in \Omega\\
\nabla \cdot \bfu(\bx) & = &0,\quad \bx\in \Omega,
\end{array}\right. ,
\end{align}
 the no-slip boundary  condition on the boundary  $\p\Omega$ is
\begin{align}
  \label{eq:bc}
\bfu=0,\quad\quad \text{at}\;\; \p\Omega.
\end{align}
For the pressure, we impose average zero condition
\begin{align}
  \label{eq:pressure-ave}
  \int_\Omega p(\bx)\mathd \bx=0.
\end{align}
The no-slip boundary condition is a Dirichlet type boundary condition and it is often used in many 
real world applications. However, the theoretical study with no-slip boundary condition is also much more difficult. 
The first question is how to enforce no-slip boundary condition in the nonlocal approach.
Recently, Du et.al. \cite{Du-SIAM} proposed volume constraint to deal with the boundary condition in the nonlocal diffusion problem by enforcing the condition over
a nonlocal region adjacent to the boundary. 
Adopting this idea, in the nonlocal Stokes system, we extend the no-slip condition to a small layer as shown in Fig. \ref{fig:domain}.
\begin{figure}
  \centering
  \includegraphics[width=0.6\textwidth]{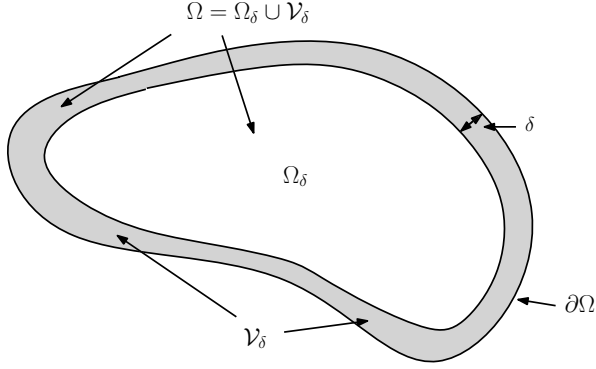}
  \caption{Computational domain in non-local Stokes model.} %\textcolor{red}{change $2\delta$ to $\delta$ to match the new scaling}}
  \label{fig:domain}
\end{figure}
For a nonlocal problem involving nonlocal interactions on the range of $\delta>0$, the whole computational domain $\Omega$ is decomposed to two parts.
%\footnote{\color{red} Instead of $\mathcal{V}_\delta$ and $ \Omega_\delta$, maybe it it simpler to use the notation $\mathcal{V}_\delta$ and $ \Omega_\delta$? Not sure what the extra $'$ does for the notation.}
$\Omega=\mathcal{V}_\delta\bigcup\Omega_\delta$ as shown in Fig. \ref{fig:domain} and $\bfu$ is enforced to be 
zero in $\V_\delta$, i.e.
\begin{equation}
\label{eq:vc}
  \bfu_\delta(\bx)=0,\quad\bx\in \mathcal{V}_\delta.
\end{equation}
Definition of $\Omega_\delta$ and $\mathcal{V}_\delta$ will be given in \eqref{eq:domain}. 
The parameter $\delta$ is often called the nonlocal horizon parameter \cite{Sill00,Du19}.
In $\Omega_\delta$, the Stokes equation is approximated is formualted as
\begin{align}
  \label{eq:stokes-nonlocal-intro}
\left\{\begin{array}{rclc}
\D \mathcal{L}_\delta \bfu_\delta(\bx) - \D \mathcal{G}_\delta p_\delta(\bx)&=&{\D \int_{\Omega} \bar{R}_\delta(\bx,\by) \bff(\by)\mathd \by},& \bx\in \Omega_\delta, \\
\D \mathcal{D}_\delta \bfu_\delta(\bx) -  \bar{\mathcal{L}}_\delta p_\delta(\bx)& = &0, & \bx\in \Omega.\\
\end{array}\right.
\end{align}
The nonlocal integral operators used in \eqref{eq:stokes-nonlocal-intro} represent the nonlocal diffusion (Laplacian) $\mathcal{L}_\delta$, nonlocal gradient $\mathcal{G}_\delta$ and nonlocal divergence $\mathcal{D}_\delta$ respectively as in \cite{DT17} and the references cited therein. 
An additional operator $\bar{\mathcal{L}}_\delta$ is also used, which is a rescaled nonlocal diffusion operator.
The particular forms of the operators adopted here are given by
\begin{align}
  \label{eq:ops-L}
  \mathcal{L}_\delta \bfu(\bx)&=\frac{1}{\delta^2}\int_\Omega R_\delta(\bx,\by)({ \bfu(\by)-\bfu(\bx)})\mathd \by,\\
\mathcal{G}_\delta p(\bx)&= \frac{1}{2\delta^2}\int_\Omega R_\delta(\bx,\by)({\by-\bx}) p(\by)\mathd \by,\\
\mathcal{D}_\delta \bfu(\bx)&=\frac{1}{2\delta^2}\int_\Omega R_\delta(\bx,\by)({ \by-\bx})\cdot \bfu(\by)\mathd \by,\\
 \bar{\mathcal{L}}_\delta p(\bx)&= \int_\Omega \bar{R}_\delta(\bx,\by)({ p(\by)-p(\bx)})\mathd \by,
\label{eq:ops-Lbar}
\end{align}
for some nonnegative and smooth kernels $R_\delta(\bx,\by)$ and $\bar{R}_\delta(\bx,\by)$ specified later.

%\footnote{\textcolor{red}{We might need $C_\delta= \alpha_n\delta^{-n}$ for some constant, not sure if we can let $\alpha_n=1$}}
%\footnote{\textcolor{red}{Could you double check? I think it might be better to define the normalization via second order moment, like in \cite{ZSD21}, otherwise the nonlocal Laplacian might be actually a scalar multiple of the local Laplacian in the local limit. This does not change the solution since the equation is scaled by the same factor.}} 

%The special form of modified integral average of $\bff$ used in  \eqref{eq:stokes-nonlocal-intro} is inspired by the point integral method \cite{SS14}.
%{ Different from \cite{SS14}, we take the average of $\bff$ over the subset $\Omega_\delta$ instead of 
%$\Omega$, which is more suitable than the original choice and allows us to complete the analysis without leaving any gaps.}

Finally, we also need average zero condition for the pressure
\begin{equation}
\label{eq:pressure-cond}
\int_{\Omega}p_\delta(\bx)\mathd\bx=0.
\end{equation}
\eqref{eq:vc}, \eqref{eq:stokes-nonlocal-intro} and \eqref{eq:pressure-cond} form a complete nonlocal formulation of the Stokes system.

As pointed out in the literature on nonlocal modeling (e.g. \cite{DT17,Du19}), nonlocal integral approximations are closely related to many numerical schemes of computational fluid dynamics, such as the smoothed particle
hydrodynamics (SPH) \cite{GM77,LL10,Lucy77,Mon05}, vortex methods \cite{BM85,CK00} and others \cite{Chertock17,CP00,ELC02,Kou05,TE04}. 
Analysis to the linear steady Stokes equation in this paper could give some new understanding to the theoretical foundation of these methods.

The Stokes system \eqref{eq:stokes} is well-known to be a saddle point problem. This remains the case for the nonlocal Stokes system given in \cite{DT17} subject to periodic boundary conditions. 
Here, different from \cite{DT17}, we add a relaxation term, $\bar{\mathcal{L}}_\delta p_\delta(\bx)$, in the second equation of \eqref{eq:stokes-nonlocal-intro}. 
{  It mimics the classical technique of stabilizing the approximation of incompressibility by adding a positive definite block to the original saddle point system. Although this results in a slightly compressible system, 
the stabilization term vanishes as $\delta\to 0$}
%The order of this term is $O(\delta)$ formally,
so that it does not destroy the approximation of the nonlocal formulation to the local limit.  Yet, this additional term is crucial 
 for { the stability and well-posedness in our case where smooth nonlocal kernels are used to define the nonlocal operators.}
 Indeed, the well-posed study in \cite{DT17} showed that, without extra relaxation, it is necessary to use singular kernels. A remedy was provided in \cite{hwi19} by incorporating non-radial nonlocal interactions.  The addition of the relaxation term enables the use of smooth kernels in the definition of the associated nonlocal operators which may allow more flexible practical implementation such as more conventional quadratures for smooth functions.
For the Fourier analysis of a related formulation with periodic boundary conditions, we refer to \cite{ZSD21}.

The rest of the paper is organized as follows. We give the { formulation of the nonlocal linear Stokes system in Section 2 together with some related assumptions and estimates}. Then the well-posedness of the
nonlocal model is established in Section 3. The vanishing nonlocality limit is analyzed in Section 4. In Section 5, we conclude with a summary and a discussion on future research.

\section{Nonlocal Stokes system with related assumptions and estimates}
\label{sec:nonlocal-stokes}
{
In this section we present the nonlocal Stokes model in more details, together with some basic assumptions on the geometry and kernel functions used to define the model, along with some related estimates.}

\subsection{Notation and assumptions}

First, we let $\Omega_\delta$ and $\mathcal{V}_\delta$ be subsets of $\Omega$ defined as
\begin{eqnarray}
  \label{eq:domain}
  \Omega_\delta=\left\{\bx\in \Omega: B\left(\bx,2\delta\right)\cap\p\Omega=\emptyset\right\},\quad \mathcal{V}_\delta=\Omega\backslash\Omega_\delta.
\end{eqnarray}
 The relation of $\Omega$, $\p\Omega$, $\Omega_\delta$ and $\mathcal{V}_\delta$ are showed in
Fig. \ref{fig:domain}.

Next, we state the following assumptions on the domain $\Omega$ and a kernel function $R(r)$. 
\begin{assumption}
\label{assumption}
\begin{itemize}
\item Assumptions on the computational domain: { $\M\in \mathbb{R}^n$ is open, bounded and connected. $\p \M$ is $C^2$ smooth.}
% \item Assumptions on the input data $(P,S,\mathbf{V})$: $(P,\mathbf{V})$ is an $h$-integrable approximation of $\M$, i.e.
% for any function $f\in C^1(\M)$, there is a constant $C$
% independent of $h$ and $f$ so that
% \begin{equation*}
% \left|\int_\M f(\by) \mathd\by - \sum_{\bx_i\in \M} f(\bx_i)V_i\right| < Ch|\text{supp}(f)|\|f\|_{C^1(\M)}.
% \end{equation*}
% where $\|f\|_{C^1} = \|f\|_\infty +\|\nabla f\|_\infty$ and
% $|\text{supp}(f)|$ is the volume of the support of $f$.

\item Assumptions on the kernel function $R(r)$:
\begin{itemize}
\item[(a)] (regularity) $R\in C^1[0,1]$;%$\footnote{\textcolor{red}{Since $R=0$ outside the unit interval, a question is if we need $R\in C^1(\mathbb{R}^+)$ or perhaps we only need to assume $R\in C^1([0,1])$?}};
\item[(b)] (positivity and compact support)
%$R(r) \le 1$ for $\forall r \in \R^+$ and $R(r) = 0$ for $\forall r >1$.
$R(r)\ge 0$ and $R(r) = 0$ for $\forall r >1$;
\item[(c)] (nondegeneracy)
 $\exists \gamma_0>0$ so that $R(r)\ge\gamma_0$ for $0\le r\le\frac{1}{2}$. \footnote{{ Here $\frac{1}{2}$ can be replaced by any constant in $(0,1)$.}}
\end{itemize}
%\item Assumptions on $t$ and $h$: $t$ and $h/\delta$ are small enough, i.e., $t$ and $h/\delta$ are less than a positive constant
%which only depends on $\M$.
\end{itemize}
\end{assumption}

Then, the rescaled kernels used in the definitions of the nonlocal operators are defined by
\begin{align}
\label{eq:kernel}
  R_\delta(\bx,\by)= C_\delta R\left(\frac{\|\bx-\by\|^2}{4\delta^2}\right),\quad 
\bar{R}_\delta(\bx,\by)= C_\delta\bar{R}\left(\frac{\|\bx-\by\|^2}{4\delta^2}\right){,}
\end{align}
where %such that 
\begin{align}
\label{eq:kernelbar}
%\int_{\mathbb{R}^n} C_\delta R\left(\frac{\|\bx\|^2}{4\delta^2}\right)\mathd \by=1, \quad\text{and}\quad
\bar{R}(r)=\int_{r}^{+\infty}R(s)\mathd s =\int_{r}^{1} R(s)\mathd s
,
\end{align}
which satisfies obviously
$$ \bar{R}'(r)=\frac{d}{dr}\bar{R}(r)=R(r),\;  \forall r\in \R^+, \quad
\text{and}\quad \bar{R}(r)=0, \;\forall r>1.$$
%The kernel function $R(r): \R^+ \rightarrow \R^+ $ is assumed to be smooth and satisfies conditions listed in the Assumption \ref{assumption} given later. 
{The constant $C_\delta=\alpha_n\delta^{-n}$ in \eqref{eq:kernel}
is a normalization factor so that
\begin{align}
\label{eq:normal-barR}
\int_{\mathbb{R}^n}\bar{R}_\delta(\bx,\by)\mathd \by=\alpha_n S_n \int_0^1 \bar{R}(\frac{r^2}{4})r^{n-1}\mathd r=1
\end{align}
with $S_n$ denotes area of the unit sphere in $\mathbb{R}^n$.
With this normalization factor,  the local limits of  {$\mathcal{L}_\delta$, $\mathcal{G}_\delta$ and
$\mathcal{D}_\delta$
recover the classical Laplacian $\Delta$, gradient and divergence operators respectively as $\delta$ goes to 0.
Moreover,  $\bar{\mathcal{L}}_\delta$ also behaves like a nonlocal analog of $\beta_n \delta^2 \Delta$, that is, a scaled nonlocal Laplacian that vanishes in the local limit.
}}

\subsection{Nonlocal Stokes system with volume constraint}
By combining the volume constraint boundary condition of $\bfu$ and the average zero condition of $p$, we have the nonlocal Stokes model given as follows:

\begin{align}
  \label{eq:stokes-nonlocal}
\left\{\begin{array}{rclc}
\D \mathcal{L}_\delta \bfu_\delta(\bx) - \D \mathcal{G}_\delta p_\delta(\bx)&=&{\D \int_{\Omega} \bar{R}_\delta(\bx,\by) \bff(\by)\mathd \by},& \bx\in \Omega_\delta, \\
\D \mathcal{D}_\delta \bfu_\delta(\bx) -  \bar{\mathcal{L}}_\delta p_\delta(\bx)& = &0, & \bx\in \Omega,\:\\[.1cm]
\bfu_\delta(\bx)&=&0,&\bx\in \mathcal{V}_\delta,\\
\D \int_{\Omega}p_\delta(\bx)\mathd\bx&=&0.
\end{array}\right.
\end{align}
% \begin{itemize}
% \item
% %$R(r) \le 1$ for $\forall r \in \R^+$ and $R(r) = 0$ for $\forall r >1$.
% $R(r) = 0$ for $\forall r >1$.
% \item
% There exists a constant $\delta_0$ so that $R(r)>\delta_0$ for $\forall r<\frac{1}{2}$.
% \end{itemize}
{
The integral operators have been defined in \eqref{eq:ops-L}-\eqref{eq:ops-Lbar}. A formal derivation of the nonlocal model is given in the appendix \ref{sec:appendix1}. 

Formally, the choices of normalization specified in this paper further imply that the local limits of  $\mathcal{L}_\delta$, $\mathcal{G}_\delta$ and
$\mathcal{D}_\delta$
recover the classical Laplacian $\Delta$, gradient and divergence operators respectively as $\delta$ goes to 0 \cite{MD16,Du19}.
Moreover,  $\bar{\mathcal{L}}_\delta$ also behaves like a nonlocal analog of $\beta_n \delta^2 \Delta$, that is, a scaled nonlocal Laplacian that vanishes in the local limit. Thus, we may see \eqref{eq:stokes-nonlocal} as a nonlocal extension of the local Stokes model \eqref{eq:stokes}-\eqref{eq:bc}.}

\begin{remark} For the study of the nonlocal model with periodic boundary condition on $\Omega=(0,1)^n$, we can use Fourier transform to get the Fourier symbols of the nonlocal operators, see the discussion in  \cite{ZSD21}. 
 \end{remark}
 
\subsection{Related estimates}

Next, we list several technical results of the kernel functions which will be used in the subsequent analysis.

{
\begin{lemma}
  \label{lem:tech}
Let  $R=R(r)$ be a kernel function satisfying Assumption \ref{assumption} and 
$R_\delta$, $\bar{R}_\delta$ be given by
\eqref{eq:kernel} and \eqref{eq:kernelbar} respectively. 
i) There exist a constant $C>0$, {independent of $\delta$},  such that
%\footnote{\color{red} there is no $C_\delta$ on right hand side?}
$$ |\nabla_{\bx}\bar{R}_\delta (\bx,\by)|\le \frac{ C}{\delta}
%R_\delta(\bx,\by)
R_\delta(\bx,\by). $$
$$ |\nabla_{\by}\nabla_{\bx}\bar{R}_\delta (\bx,\by)|\le  \frac{C}{\delta^2} 
\left(|R'_\delta(\bx,\by)
|+|R_\delta(\bx,\by)|\right),$$
for any $\bx, \by\in \mathbb{R}^n$, where
%where $C_\delta=\alpha_n \delta^{-n}$ is the normalization constant specified in \eqref{eq:normal-barR}
%, $S_n$ is the area of the unit sphere in $\mathbb{R}^n$  
%and %
$R'_\delta(\bx,\by)=C_\delta R'\left(\frac{\|\bx-\by\|^2}{4\delta^2}\right)$ and
$R'(r)=\frac{\mathd R(r)}{\mathd r}$;

ii)
 Let $\tilde{R}$ be a kernel function
 satisfying the Assumption \ref{assumption} (a) (b) and $\tilde{R}_\delta (\bx,\by)=\alpha_n\delta^{-n} \tilde{R}\left(\frac{\|\bx-\by\|^2}{4\delta^2}\right)$. 
 There exists a constant $\eta_0>0$ only dependent on $\Omega$
{and $\tilde{R}$}, such that for $\delta\le \eta_0$
$$\frac{\tilde{\omega}_n}{3}
%\alpha_n S_n \int_0^1 \tilde{R}(r)r^{n-1}\mathd r
<   {\tilde{\omega}_\delta(\bx):=}
\int_\Omega \tilde{R}_\delta(\bx,\by)\mathd \by \le \tilde{\omega}_n:=
\alpha_n S_n \int_0^1 \tilde{R}(\frac{r^2}{4})r^{n-1}\mathd r, $$

iii) Let $$K_\delta(\by,\bz)=\int_\Omega|\tilde{R}_\delta(\bx,\bz)\nabla_{\bx} \tilde{R}_\delta(\bx,\by)|\mathd \bx.$$
{for any  $\by, \bz\in \mathbb{R}^n$.}
There exist $C>0$ independent on $\delta$ such that 
$$K_\delta(\by,\bz)\le C{R}\left(\frac{\|\by-\bz\|^2}{32\delta^2}\right),$$
%\footnote{\color{red}for what $\by$ and $\bz$? $\bz$ should be $\bx$?}
\end{lemma}

\begin{proof}
i) can be checked directly.\\
ii). { This estimate is classical for smooth mollifiers. For the sake of completeness, we give a brief proof here.}
%\footnote{{\color{red} This should be classical estimates on mollifiers, so probbaly can get references without proofs}}
%Right inequality
{The upper bound}
is easy to prove using the %positivity
{non-negativity
of $\tilde{R}_\delta$}.
\begin{align*}
   {\tilde{\omega}_\delta(\bx) =} \int_\Omega  { \tilde{R}}_\delta(\bx,\by)\mathd \by\le \int_{\mathbb{R}^n} { \tilde{R}}_\delta(\bx,\by)\mathd \by = { \tilde{\omega}_n}
    %\alpha_n S_n \int_0^1 R(r)r^{n-1}\mathd r
\end{align*}
To prove the %left inequality
lower bound, we need to use the condition that $\partial \Omega$ is $C^2$
{ and $\tilde{R}_\delta$ is continuous and bounded}. Then for $\bx\in \partial \Omega$,
$$ {\lim_{\delta\rightarrow 0} \tilde{\omega}_\delta(\bx)=}
\lim_{\delta\rightarrow 0} \int_\Omega { \tilde{R}}_\delta(\bx,\by)\mathd \by =
{\alpha_n}
\int_{
{\bx+}\mathbb{R}^n_+} { \tilde{R}\left(\frac{\|\bx-\by\|^2}{4}\right)}\mathd \by =\frac{\tilde{\omega}_n}{2},
%\alpha_n S_n \int_0^1 R(r)r^{n-1}\mathd r, 
$$
where { $\mathbb{R}^n_+=\{\by=(y_1,\cdots,y_n)\in \mathbb{R}^n: y_1\ge 0\}$.}

On the other hand, for $\bx\in {\Omega} $, since $\Omega$ is open, %\footnote{
%\textcolor{red}{$\Omega$ is generally referring to an open set so 
%$\mathring{\Omega}=\Omega$
%}}
$$ {\lim_{\delta\rightarrow 0} \tilde{\omega}_\delta(\bx)=}
\lim_{\delta\rightarrow 0} \int_\Omega { \tilde{R}}_\delta(\bx,\by)\mathd \by =
{\alpha_n} \int_{\mathbb{R}^n} { \tilde{R}\left(\frac{\|\bx-\by\|^2}{4}
\right) }\mathd \by = \tilde{\omega}_n
%\alpha_n S_n \int_0^1 R(r)r^{n-1}\mathd r
. $$
So, for any $\bx\in \bar{\Omega}=\Omega\cup \partial \Omega$, there exist $\eta_{\bx}>0$ such that for any $\delta\le \eta_{\bx}$, we have
$ {
\tilde{\omega}_\delta(\bx)}>{\tilde{\omega}_n}/{3}
%\alpha_n S_n \int_0^1 R(r)r^{n-1}\mathd r
$.
Using the compactness of $\bar{\Omega}$, there exists $\eta_0>0$ such that for any $\bx\in \bar{\Omega}$, $\delta\le \eta_{0}$, we have
$ {
\tilde{\omega}_\delta(\bx)}>{\tilde{\omega}_n}/{3}
%\alpha_n S_n \int_0^1 R(r)r^{n-1}\mathd r
$.

ii) When { $\|\by-\bz\|\ge 4\delta$}, we have $\max\{\|\bx-\bz\|, \|\bx-\by\|\}\ge 2\delta$, then using condition (a) and (b) in Assumption \ref{assumption},
$$\tilde{R}_\delta(\bx,\bz)\nabla_{\bx} \tilde{R}_\delta(\bx,\by)=0,\quad \forall \bx\in \Omega.$$
This gives that
$$K_\delta(\by,\bz)=0,\quad \forall \|\by-\bz\|\ge 4\delta.$$

%$$R_{2\delta}(\bx,\by)\ge 0 =\left|\tilde{R}'_\delta(\bx,\by)\right|$$

%using\footnote{\color{red} you mean $|\bx-\by|\le \delta$ for $R_{2\delta}$ positive? Anyway, the $4\delta^2$
%needs to be corrected.} 

When { $\|\by-\bz\|< 4\delta$}, we have $\frac{\|\by-\bz\|^2}{32\delta^2}<\frac{1}{2}$. Using
condition (c) in Assumption \ref{assumption}, 
\begin{align*}
    K_\delta(\by,\bz)=&\int_\Omega|\tilde{R}_\delta(\bx,\bz)\nabla_{\bx} \tilde{R}_\delta(\bx,\by)|\mathd \bx\\
    \le & \frac{1}{4\delta^2}\int_\Omega \|\bx-\by\||\tilde{R}_\delta(\bx,\bz)||\tilde{R}'_\delta(\bx,\by)|\mathd \bx\\
    \le& \frac{1}{2\delta}\int_\Omega |\tilde{R}_\delta(\bx,\bz)||\tilde{R}'_\delta(\bx,\by)|\mathd \bx\\
    \le & \frac{C_\delta^2}{2\delta}\int_{\Omega\cap B(\frac{\by+\bz}{2},2\delta)} |\tilde{R}\left(\frac{\|\bx-\bz\|^2}{4\delta^2}\right)||\tilde{R}'\left(\frac{\|\bx-\bz\|^2}{4\delta^2}\right)|\mathd \bx\\
    \le & \frac{\tilde{M} C_\delta^2}{2\delta}\left| B(\frac{\by+\bz}{2},2\delta)\right|\\
    \le & \frac{C \tilde{M}}{\delta\gamma_0} C_\delta\gamma_0
    \le  \frac{C \tilde{M}}{\delta\gamma_0} C_\delta R\left(\frac{\|\by-\bz\|^2}{32\delta^2}\right)
\end{align*}
where $\tilde{M}=\max_{r\in [0,1]} |\tilde{R}(r)\tilde{R}'(r)|$, $\gamma_0$ is the constant in condition (c) in Assumption \ref{assumption}, $C_\delta$ is the normalization factor in \eqref{eq:kernel}. %\footnote{\textcolor{red}{But you are only assuming (a) and (b) for  $\tilde{R}$ in the lemma.} {\color{blue} The right hand side is $R$. $R$ is assumed to satisfy whole assumption.}} 
\end{proof}
}

 \section{Well-posedness of the nonlocal Stokes system \eqref{eq:stokes-nonlocal}}
 \label{sec:wellpose}

In this section, we prove the well-posedness of the nonlocal Stokes system \eqref{eq:stokes-nonlocal}. More precisely, we show the following theorem.

\begin{theorem}
  \label{thm:wellpose}
Suppose that the Assumption \ref{assumption} is satisfied. For any { $\bff\in H^{-1}(\Omega)$}, there exits one and only one pair $(\bfu, p)$, such that
\begin{itemize}
\item[(a)] $\bfu\in H^1(\Omega_\delta)$, $p\in L^2(\Omega)$. In addition,
  \begin{align*}
    \|\bfu\|_{H^1(\Omega_\delta)}+\|p\|_{L^2(\Omega)}\le C { \|\bff\|_{H^{-1}(\Omega)}},
  \end{align*}
where $C>0$ is a constant that only depends on $\Omega$ and kernel function $R$.
\item[(b)] The pair $(\bfu,p)$ satisfies the nonlocal Stokes system \eqref{eq:stokes-nonlocal}.
\end{itemize}
\end{theorem}
In the proof of the well-posedness, we need several technical lemmas.

\begin{lemma} (\cite{SS14}) 
If $\delta$ is small enough, for any function
%\footnote{\color{red} Change $M$ to $\Omega$ or $\Omega_\delta$? If latter, 
%can we quote [36] directly? to say $\delta$ independent of $\Omega_\delta$? {\color{blue} $M$ is changed to $\Omega$.}}
$u\in L^2(\Omega)$,
there exists a constant $C>0$, independent of $\delta$ and $u$, such that
\begin{eqnarray*}
   \int_\M\int_{\M}R\left(\frac{\|\bx-\by\|^2}{32\delta^2}\right) (u(\bx)-u(\by))^2\mathd \bx \mathd \by \le C
\int_\M\int_{\M}R\left(\frac{\|\bx-\by\|^2}{4\delta^2}\right) (u(\bx)-u(\by))^2\mathd \bx \mathd \by.  \nonumber
\end{eqnarray*}
\label{lem:bigt2smallt}
\end{lemma}

{Similar results concerning the scaling of the nonlocal interaction neighborhood  like the above one
can also be found in  \cite{DTWY21} for other types of kernels including fractional ones.}

{Next, we consider an extension to a similar result shown in \cite{Shi-vc}.
%\footnote{\color{red} It might be good to make a remark on how large $\int_{\V_\delta} R_\delta(\bx,\by)\mathd \by$ is wrt $\delta$, depending on the position $\bx$.}
.
\begin{lemma} %(\cite{Shi-vc})
\label{cor:coercivity-inner}
For any function {$u\in L_2(\mathbb{R}^n)$ and vanish outside $\Omega_\delta$, i.e. $u(\bx)=0$ for $\bx\in \mathbb{R}^n\backslash \Omega_\delta$},
there exists a constant $C>0$ independent on $\delta$, such that
 \begin{eqnarray*}
  &&  \frac{1}{\delta^2}\int_{\Omega_\delta}\int_{\Omega_\delta} R_\delta(\bx,\by)(u(\bx)-u(\by))^2\mathd\bx\mathd\by
+\frac{1}{\delta^2}\int_{\M_\delta}u^2(\bx)\left(\int_{\V_\delta} R_\delta(\bx,\by)\mathd \by\right)\mathd\bx \nonumber\\
&&\quad \ge C\int_{\Omega} |\nabla v|^2\mathd \bx,\nonumber
  \end{eqnarray*}
where
\begin{eqnarray}
v(\bx)=\frac{1}{\tilde{w}_\delta(\bx)}\int_{\Omega_\delta}\tilde{R}_\delta(\bx,\by)u(\by)\mathd \by=\frac{1}{\tilde{w}_\delta(\bx)}\int_{\Omega}\tilde{R}_\delta(\bx,\by)u(\by)\mathd \by, \nonumber
\end{eqnarray}
and 
$$\D \tilde{w}_\delta(\bx) = \int_{\Omega}\tilde{R}_\delta\left(\bx,\by\right)\mathd \by, \quad \text{where}\quad
\tilde{R}_\delta(\bx,\by)=C_\delta\tilde{R}\left(\frac{|\bx-\by|^2}{4\delta^2}\right),
$$ 
and
$\tilde{R}$ is a kernel function satisfying condition (a)-(b)-(c) in Assumption \ref{assumption}. %{\color{black} }
\end{lemma}

\begin{proof}
For any $\bx\in \M$, we have
\begin{align*}
  \nabla v(\bx)=&\frac{1}{\tilde{w}_\delta(\bx)}\int_{\Omega}\nabla_{\bx} \tilde{R}_\delta(\bx,\by){u}(\by)\mathd \by-\frac{\nabla \tilde{w}_\delta(\bx)}{\tilde{w}^2_\delta(\bx)}\int_{\Omega}( \tilde{R}_\delta(\bx,\by)){u}(\by)\mathd \by\\
  =& \frac{1}{\tilde{w}^2_\delta(\bx)}\int_{\Omega}\int_{\Omega}\tilde{R}_\delta(\bx,\bz)\nabla_{\bx} \tilde{R}_\delta(\bx,\by)({u}(\by)-{u}(\bz))\mathd \by\mathd \bz
\end{align*}
This leads to
%the bound of $\int_{\Omega_\delta} |\nabla v(\bx)|^2\mathd\bx$ as follows,
\begin{align*}
  &\int_{\Omega} |\nabla v(\bx)|^2\mathd\bx\\
  &\quad =  \int_\Omega \frac{1}{\tilde{w}^4_\delta(\bx)}\left|\int_{\Omega}\int_{\Omega}\tilde{R}_\delta(\bx,\bz)\nabla_{\bx} \tilde{R}_\delta(\bx,\by)(u(\by)-u(\bz))\mathd \by\mathd \bz\right|^2\mathd \bx\\
   &\quad \le  \frac{1}{\tilde{\omega}^4_{\min}}\int_\Omega\left(\int_{\Omega}\int_{\Omega}|\tilde{R}_\delta(\bx,\bz)\nabla_{\bx} \tilde{R}_\delta(\bx,\by)|\mathd \by\mathd \bz\right)\\
   & \qquad\quad \left( \int_{\Omega}\int_{\Omega}|\tilde{R}_\delta(\bx,\bz)\nabla_{\bx} \tilde{R}_\delta(\bx,\by)|(u(\by)-u(\bz))^2\mathd \by\mathd \bz \right)  \mathd \bx\\
   &\quad \le  \frac{C}{\delta\tilde{\omega}^4_{\min}}\int_{\Omega}\int_{\Omega}K_\delta(\by,\bz)(u(\by)-u(\bz))^2\mathd \by\mathd \bz
\end{align*}
with $\tilde{\omega}_{\min}=\frac{1}{3}\alpha_n S_n\int_0^1 \tilde{R}(\frac{r^2}{4})r^{n-1}\mathd r$ given in Lemma \ref{lem:tech} and
%\footnote{\color{red} The result in Lemma \ref{lem:tech} was on 
%${R}$ not $\tilde{R}$. {\color{blue} $\tilde{R}$ also satisfies condition (a) (b) in Assumption 1. So Lemma 2.1 also applies.}
%},
 
$$K_\delta(\by,\bz)=\int_\Omega|\tilde{R}_\delta(\bx,\bz)\nabla_{\bx} \tilde{R}_\delta(\bx,\by)|\mathd \bx.$$
In the last inequality, we use the following estimate,
\begin{align*}
  &\int_{\Omega}\int_{\Omega}\left|\tilde{R}_\delta(\bx,\bz)\right|\left|\nabla_{\bx} \tilde{R}_\delta(\bx,\by)\right|\mathd \by\mathd \bz\\
  \le&\frac{1}{4\delta^2}\int_\Omega \int_{\Omega}\|\bx-\by\|\left|\tilde{R}'_\delta(\bx,\by)\right|\left|\tilde{R}_\delta(\bx,\bz)\right|\mathd \by\mathd \bz\\
  \le&\frac{1}{2\delta}\int_\Omega \int_{\Omega}\left|\tilde{R}'_\delta(\bx,\by)\right|\left|\tilde{R}_\delta(\bx,\bz)\right|\mathd \by\mathd \bz
  \le  \frac{C}{\delta}
\end{align*}
with%\footnote{\color{red} missing $C_\delta$?}
$\tilde{R}'_\delta(\bx,\by)={C_\delta\tilde{R}'\left(\frac{|\bx-\by|^2}{4\delta^2}\right)}$ and $\tilde{R}'(r)=\frac{\mathd \tilde{R}}{\mathd r}$.

Finally, Lemma \ref{lem:tech} ii) gives that
\begin{align*}
  &\int_{\Omega_\delta} |\nabla v(\bx)|^2\mathd\bx\\
\le & \; \frac{C}{\delta^2}\int_{\Omega}\int_{\Omega} C_\delta R\left(\frac{|\bx-\by|^2}{32\delta^2}\right)(u(\bx)-u(\by))^2\mathd \bx\mathd\by\\
\le & \; \frac{C}{\delta^2}\int_{\Omega}\int_{\Omega} C_\delta R\left(\frac{|\bx-\by|^2}{4\delta^2}\right)(u(\bx)-u(\by))^2\mathd \bx\mathd\by\\
=&\; \frac{C}{\delta^2}\int_{\Omega_\delta}\int_{\Omega_\delta} R_\delta(\bx,\by)(u(\bx)-u(\by))^2\mathd \bx\mathd\by+
\frac{2C}{\delta^2}\int_{\V_\delta}\left(\int_{\Omega_\delta} R_\delta(\bx,\by)u(\by)^2\mathd \by\right)\mathd\bx
\end{align*}
The second inequality comes from Lemma \ref{lem:bigt2smallt}.
%\footnote{\color{red} In the previous step, need to control
%$\left|\nabla_{\bx} R_\delta(\bx,\by)\right|$
%by $R_\delta(\bx,\by$? Also, where is  $\bar{w}_\delta$ defined? Maybe we should explain what $C$ represents to show clearly that it is independent of $\delta$}
\end{proof}

Using above Lemma, it is easy to get a nonlocal Poinc\'{a}re inequality for the special kernels, Lemma \ref{cor:l2-inner}.
\begin{lemma}
  \label{cor:l2-inner}
 For any function {$u\in L_2(\mathbb{R}^n)$ and vanish outside $\Omega_\delta$},
there exists a constant $C>0$ independent on $\delta$, such that
\begin{eqnarray}
\frac{1}{\delta^2}\int_{\M_\delta}\int_{\M_\delta}R_\delta(\bx,\by)(u(\bx)-u(\by))^2\mathd\bx\mathd\by
+{\color{black}\frac{1}{\delta^2}}\int_{\M_\delta}u^2(\bx)\left(\int_{\V_\delta}R_\delta(\bx,\by)\mathd \by\right)\mathd\bx
\geq C\|u\|_{L_2(\M_\delta)}^2,\nonumber
\end{eqnarray}
as long as $\delta$ small enough.
\end{lemma}
\begin{proof}
Let 
$$v(\bx)=\frac{1}{w_\delta(\bx)}\int_{\Omega_\delta}R_\delta(\bx,\by)u(\by)\mathd \by=\frac{1}{w_\delta(\bx)}\int_{\Omega}R_\delta(\bx,\by)u(\by)\mathd \by$$
Using the definition of $\Omega_\delta$, 
$$v(\bx)=0,\quad \forall \bx\in \partial \Omega.$$
Then Lemma \ref{cor:coercivity-inner} and Poinc\'{a}re inequality imply that
$$\|v\|^2_{L^2(\Omega)}\le \frac{C}{\delta^2}\left(\int_{\M_\delta}\int_{\M_\delta}R_\delta(\bx,\by)(u(\bx)-u(\by))^2\mathd\bx\mathd\by
+\int_{\M_\delta}u^2(\bx)\left(\int_{\V_\delta}R_\delta(\bx,\by)\mathd \by\right)\mathd\bx\right).$$
On the other hand, for $\bx\in \Omega_\delta$
\begin{align*}
    u(x)-v(x)=\frac{1}{w_\delta(\bx)}\int_{\Omega}R_\delta(\bx,\by)(u(\bx)-u(\by))\mathd \by
\end{align*}
such that
\begin{align*}
    &\|u-v\|_{L^2(\Omega_\delta)}^2\\
    \le& \frac{1}{\omega_{\min}^2}\int_\Omega\left(\int_{\Omega}R_\delta(\bx,\by)\mathd \by\right)\int_{\Omega}R_\delta(\bx,\by)(u(\bx)-u(\by))^2\mathd \by\mathd 
    \bx\\
    \le &\frac{\omega_{\max}}{\omega_{\min}^2}\int_\Omega\int_{\Omega}R_\delta(\bx,\by)(u(\bx)-u(\by))^2\mathd \by\mathd 
    \bx\\
    =& \frac{\omega_{\max}}{\omega_{\min}^2}\left(\int_{\M_\delta}\int_{\M_\delta}R_\delta(\bx,\by)(u(\bx)-u(\by))^2\mathd\bx\mathd\by
+2\int_{\M_\delta}u^2(\bx)\left(\int_{\V_\delta}R_\delta(\bx,\by)\mathd \by\right)\mathd\bx\right)
 \end{align*}
 where $\omega_{\min} =\frac{1}{3}\alpha_n S_n\int_0^1 R(\frac{r^2}{4})r^{n-1}\mathd r$ and $\omega_{\max} =\alpha_n S_n\int_0^1 R(\frac{r^2}{4})r^{n-1}\mathd r$ as given in Lemma \ref{lem:tech}.
\end{proof}
\begin{remark}
Support of $\int_{\V_\delta}R_\delta(\bx,\by)\mathd \by$ is a narrow band adjacent to $\partial \Omega$ with the width of $4\delta$. So the second term in Lemma \ref{cor:l2-inner}, $\frac{1}{\delta^2}\int_{\M_\delta}u^2(\bx)\left(\int_{\V_\delta}R_\delta(\bx,\by)\mathd \by\right)\mathd\bx$, is used to control $u(\bx)$ near the boundary while the first term controls the fluctuation in the interior. Lemma \ref{cor:l2-inner} is actually very natural following the spirit of the Poinc\'{a}re inequality. For more general discussions, we refer to, e.g., \cite{MD14,MD16,Du19} and the references cited therein.
\end{remark}
}
\begin{lemma}
  \label{cor:l2-inner-ave}
 For any function {$p\in L_2(\Omega)$ with $\int_{\Omega_\delta}p(\bx)\mathd \bx=0$},
there exists a constant $C>0$ independent on $\delta$, such that
\begin{eqnarray}
\frac{1}{\delta^2}\int_{\M}\int_{\M}\bar{R}_\delta(\bx,\by)(p(\bx)-p(\by))^2\mathd\bx\mathd\by
\geq C\|p\|_{L_2(\M)}^2,\nonumber
\end{eqnarray}
as long as $\delta$ small enough.
\end{lemma}
\begin{proof}
For $p$ with $\int_{\Omega_\delta}p(\bx)\mathd \bx=0$, we also have nonlocal Poinc\'{a}re inequality \cite{SS14},
\begin{align*}
    \|p\|_{L^2(\Omega_\delta)}^2\le \frac{C}{\delta^2}\int_{\Omega_\delta}\int_{\Omega_\delta}\bar{R}_{\delta}(\bx,\by)(p(\bx)-p(\by))^2\mathd \bx\mathd \by
\end{align*}
Using nondegeneracy assumption in Assumption \ref{assumption}, it is easy to verify that for any $\bx\in \Omega$,
\begin{equation}
%\label{eq:low-4d}
  \int_{\Omega_\delta} \bar{R}_{4\delta}(\bx,\by)\mathd\by\ge c_0>0.\nonumber
\end{equation}
where 
$$\bar{R}_{4\delta}(\bx,\by)=C_\delta \bar{R}\left(\frac{\|\bx-\by\|^2}{4(4\delta)^2}\right),$$
and $C_\delta$ is the normalization factor in \eqref{eq:kernel}.
\begin{align}
%\label{eq:p-inner-global-wp}
 & \|p\|_{L^2(\Omega)}^2\nonumber\\
  \le& C\int_\Omega|p(\bx)|^2\left(\int_\Omega R_\delta(\bx,\by)\mathd\by\right)\mathd\bx\nonumber\\
\le&C\int_\Omega\left(\int_{\Omega_\delta}|p(\bx)|^2 \bar{R}_{4\delta}(\bx,\by)\mathd\by\right)\mathd\bx\nonumber\\
\le&C\int_\Omega\left(\int_{\Omega_\delta}|p(\bx)-p(\by)|^2 \bar{R}_{4\delta}(\bx,\by)\mathd\by\right)\mathd\bx\nonumber\\
&+
C\int_\Omega\left(\int_{\Omega_\delta}|p(\by)|^2 \bar{R}_{4\delta}(\bx,\by)\mathd\by\right)\mathd\bx\nonumber\\
\le&C\int_\Omega\int_{\Omega}|p(\bx)-p(\by)|^2 \bar{R}_{4\delta}(\bx,\by)\mathd\by\mathd\bx+
C\int_{\Omega_\delta}|p(\by)|^2\mathd\bx\nonumber\\
\le&C\int_\Omega\int_{\Omega}|p(\bx)-p(\by)|^2 \bar{R}_{\delta}(\bx,\by)\mathd\by\mathd\bx+
C\|p\|_{L^2(\Omega_\delta)}^2\nonumber
\end{align}
\end{proof}

Now we can prove the main theorem in this section, Theorem \ref{thm:wellpose}.
\begin{proof}[Proof of Theorem \ref{thm:wellpose}:]

First, in the nonlocal Stokes system, we replace the condition $$\int_\Omega p_\delta(\bx)\mathd \bx=0$$
by 
$$\int_{\Omega_\delta} p_\delta(\bx)\mathd \bx=0$$
and denote the 
pressure in the original nonlocal Stokes system as $\bar{p}_\delta$.
 It is obvious that 
\begin{align}
\label{eq:p-pbar}
\bar{p}_\delta=p_\delta-\frac{1}{|\Omega|}\int_\Omega p_\delta(\bx)\mathd\bx.
\end{align}
 { The existence and uniqueness of the solution to the nonlocal Stokes system is a direct implication of Lax-Milgram Theorem
by introducing the bilinear form in $\bm{L}_\delta^2(\Omega)\times \bm{L}_\delta^2(\Omega)$:
\begin{align}
\label{eq:bilinear}
a([\bfu,p],[\bm{v},q])=&\frac{1}{2\delta^2}\int_\Omega\int_\Omega R_\delta(\bx,\by)(\bfu(\bx)-\bfu(\by))\cdot(\bm{v}(\bx)-\bm{v}(\by))\mathd \bx\mathd \by\\
&+\frac{1}{2\delta^2}\int_\Omega\int_\Omega R_\delta(\bx,\by)(\by-\bx)\cdot( \bm{v}(\bx)p(\by)-\bfu(\bx)q(\by))\mathd \bx\mathd \by\nonumber\\
&+\int_\Omega\int_\Omega \bar{R}_\delta(\bx,\by)(p(\bx)-p(\by))(q(\bx)-q(\by))\mathd \bx\mathd \by\nonumber
\end{align}
where 
$$\bm{L}_\delta^2(\Omega)=\left\{[\bfu,p]:  \bfu\in L^2(\Omega)^n, p\in L^2(\Omega), \mbox{supp}(\bfu)\subset \Omega_\delta, \int_{\Omega_\delta}p(\bx)\mathd \bx=0\right\}.$$

To apply Lax-Milgram Theorem, we need to check the continuity and coercivity of the bilinear form, i.e. for any $[\bfu,p], [\bm{v},q]\in \bm{L}_\delta^2(\Omega)$,
\begin{align*} 
    |a([\bfu,p],[\bm{v},q])|\le C( \|\bfu\|_{L^2(\Omega)}
    +\|p\|_{L^2(\Omega)})( \|\bm{v}\|_{L^2(\Omega)}
    +\|q\|_{L^2(\Omega)})
\end{align*}
and
\begin{align*}
    a([\bfu,p],[\bfu,p])\ge C( \|\bfu\|_{L^2(\Omega)}^{ 2}
    +\|p\|_{L^2(\Omega)}^{2})
\end{align*}
with $C>0$ independent on $[\bfu,p]$ and $[\bm{v},q]$.\footnote{ Here constant $C$ may depend on $\delta$.}  

{The continuity is easy to check and coercivity can be given by Lemma \ref{cor:l2-inner}, Lemma \ref{cor:l2-inner-ave}}.
%Lemma \ref{cor:l2-inner} immediately gives that
%\begin{align*}
%    a([\bfu,p],[\bfu,p])\ge C \|\bfu\|_{L^2(\Omega)}^2
%\end{align*}
Then, the existence and uniqueness of the solution is given by Lax-Milgram Theorem {(Section 6.2.1 in \cite{Evans})}.

 In the rest of the proof, we will devote to get the uniform upper bound of $\|\bfu\|_{H^1(\Omega_\delta)}^{{ 2}}+\|p\|_{L^2(\Omega)}^{{2}}$.}
Multiplying $\bfu_\delta$ on the first equation of \eqref{eq:stokes-nonlocal} and multiplying $p$ on the second equation of \eqref{eq:stokes-nonlocal} and integrating over $\Omega$ and adding them together,
we can get
\begin{align}
\label{eq:u-0}
 & \frac{1}{\delta^2}\int_\Omega\int_\Omega R_\delta(\bx,\by)|\bfu_\delta(\bx)-\bfu_\delta(\by)|^2\mathd \bx\mathd \by+\int_\Omega\int_\Omega \bar{R}_\delta(\bx,\by)(p_\delta(\bx)-p_\delta(\by))^2\mathd \bx\mathd \by\\
&\qquad =-2{\int_{\Omega}\left(\int_{\Omega} \bar{R}_\delta(\bx,\by)\bff(\by)\mathd \by\right)\cdot\bfu_\delta(\bx)\mathd \bx}\nonumber
\end{align}

{
From \eqref{eq:u-0}, using Lemma \ref{cor:l2-inner}, we have
\begin{align}
\|\bm{u}_\delta\|_{L^2(\Omega)}^2\le C\int_{\Omega}\left(\int_{\Omega} \bar{R}_\delta(\bx,\by)\bff(\by)\mathd \by\right)\cdot\bfu_\delta(\bx)\mathd \bx.\nonumber
\end{align}

and
\begin{align}
\left|\int_{\Omega}\int_{\Omega} \bar{R}_\delta(\bx,\by)\bff(\by)\cdot\bfu_\delta(\bx)\mathd \bx\mathd \by\right|\le C \|\bff\|_{H^{-1}(\Omega)}\|\tilde{\bfu}_\delta\|_{H^1(\Omega)}.\nonumber
\end{align}
with
$$\tilde{\bfu}_\delta(\by)=\int_{\Omega} \bar{R}_\delta(\bx,\by)\bfu_\delta(\bx)\mathd \bx.$$
 %Notice that for $\by\in \Omega_\delta$, $\bar{\omega}_\delta(\by)=\int_{\Omega} \bar{R}_\delta(\bx,\by)\mathd \bx=\alpha_n S_n\int_0^1 \bar{R}(\frac{r^2}{4})r^{n-1}\mathd r$ is a constant only depends on $\bar{R}$. 
 Notice that
 $\bfu_\delta(\by)=0,\; \by\in \mathcal{V}_\delta$ and
 $\int_\Omega \nabla_{\by}\bar{R}_\delta(\bx,\by)\mathd \bx=0,\; \by\in \Omega_\delta,$
 so 
 $$\bfu_\delta(\by)\int_\Omega \nabla_{\by}\bar{R}_\delta(\bx,\by)\mathd \bx=0,\quad \by\in \Omega.$$
Then we have
 \begin{align*}
     \|\nabla \tilde{\bfu}_\delta\|_{L^2(\Omega)}^2=&\int_\Omega \left|\int_{\Omega} \nabla_{\by}\bar{R}_\delta(\bx,\by)\bfu_\delta(\bx)\mathd \bx\right|^2\mathd \by\\
     =&\int_\Omega \left|\int_{\Omega} \nabla_{\by}\bar{R}_\delta(\bx,\by)(\bfu_\delta(\bx)-\bfu_\delta(\by))\mathd \bx\right|^2\mathd \by\\
     \le&\int_\Omega \left(\int_{\Omega} |\nabla_{\by}\bar{R}_\delta(\bx,\by)|\mathd \bx\right) \int_{\Omega} |\nabla_{\by}\bar{R}_\delta(\bx,\by)||\bfu_\delta(\bx)-\bfu_\delta(\by)|^2\mathd \bx\mathd \by\\
     \le & \frac{C}{\delta^2}\int_\Omega \int_{\Omega} R_\delta(\bx,\by)|\bfu_\delta(\bx)-\bfu_\delta(\by)|^2\mathd \bx\mathd \by
 \end{align*}
 to get the last inequality, Lemma \ref{lem:tech} is used. 
 
 {
Moreover, it is easy to see that 
$$\|\tilde{\bfu}_\delta\|_{L^2(\Omega)}^2\le C\|{\bfu}_\delta\|_{L^2(\Omega)}^2\le \frac{C}{\delta^2}\int_\Omega \int_{\Omega} R_\delta(\bx,\by)|\bfu_\delta(\bx)-\bfu_\delta(\bx)|^2\mathd \bx\mathd \by.$$
Putting above estimates together, we have
\begin{align*}
   \|\tilde{\bfu}_\delta\|_{H^1(\Omega)}^2\le&  \frac{C}{\delta^2}\int_\Omega\int_\Omega R_\delta(\bx,\by)|\bfu_\delta(\bx)-\bfu_\delta(\by)|^2\mathd \bx\mathd \by\\
   \le& C \left|\int_{\Omega}\int_{\Omega} \bar{R}_\delta(\bx,\by)\bff(\by)\cdot\bfu_\delta(\bx)\mathd \bx\mathd \by\right|
\end{align*}
It follows that
\begin{align*}
 \left|\int_{\Omega}\int_{\Omega} \bar{R}_\delta(\bx,\by)\bff(\by)\cdot\bfu_\delta(\bx)\mathd \bx\mathd \by\right|\le  C  \|\bff\|^2_{H^{-1}(\Omega)}
\end{align*}
}

%there exists $C>0$, such that
Hence, we get
\begin{align}
  \label{eq:u-l2}
\|\bm{u}_\delta\|_{L^2(\Omega)}\le C\|\bff\|_{H^{-1}(\Omega)}.
\end{align}
}
In addition, from the first equation of \eqref{eq:stokes-nonlocal}, $\bfu_\delta$ has following expression, for any $\bx\in \Omega_\delta$,
\begin{align}
 \label{eq:u-h1-0}
  \bfu_\delta(\bx)=&\frac{1}{w_\delta(\bx)}\int_{\Omega} R_\delta(\bx,\by)\bfu_\delta(\by)\mathd \by+\frac{1}{2w_\delta(\bx)}\int_\Omega R_\delta(\bx,\by)(\bx-\by)p_\delta(\by)\mathd\by
  \nonumber\\
  &
-\frac{\delta^2}{w_\delta(\bx)}
{ \int_{\Omega}} \bar{R}_\delta(\bx,\by)\bff(\by)\mathd\by
\end{align}
Using Lemma \ref{cor:coercivity-inner}, \eqref{eq:u-0} and \eqref{eq:u-l2}, we have
\begin{align}
  \label{eq:u-h1-1}
&  \|\nabla \left(\frac{1}{w_\delta(\bx)}\int_{\Omega} R_\delta(\bx,\by)\bfu_\delta(\by)\mathd \by\right)\|_{L^2(\Omega_\delta)}^2\\
\le&\frac{C}{\delta^2}\int_{\Omega}\int_{\Omega}R_\delta(\bx,\by)|\bfu_\delta(\bx)-\bfu_\delta(\by)|^2\mathd\bx\mathd\by\nonumber\\
\le& C
{ \|\bff\|_{H^{-1}(\Omega)}^2}\nonumber
\end{align}

Notice that for any $\bx\in \Omega_\delta$, $w_\delta(\bx)$ is a positive constant. Then we have
\begin{align}
 \label{eq:u-h1-2}
  &\|\nabla\left(\frac{1}{2w_\delta(\bx)}\int_\Omega R_\delta(\bx,\by)(\bx-\by)p_\delta(\by)\mathd\by\right)\|_{L^2(\Omega_\delta)}^2\\
\le& C\int_{\Omega_\delta}\left|\int_{\Omega}\nabla_{\bx} R_\delta(\bx,\by)(\bx-\by)p_\delta(\by)\mathd\by\right|^2\mathd\bx+
C\int_{\Omega_\delta}\left(\int_{\Omega} R_\delta(\bx,\by)p_\delta(\by)\mathd\by\right)^2\mathd\bx\nonumber\\
\le& \frac{C}{\delta^2}\int_{\Omega}\left|\int_{\Omega}|R'_\delta(\bx,\by)||\bx-\by|^2 |p_\delta(\by)|\mathd\by\right|^2\mathd\bx+
C\int_{\Omega}\left(\int_{\Omega} R_\delta(\bx,\by)p_\delta(\by)\mathd\by\right)^2\mathd\bx\nonumber\\
\le& C\int_{\Omega}\left(\int_{\Omega}|R'_\delta(\bx,\by)| |p_\delta(\by)|\mathd\by\right)^2\mathd\bx+C\int_{\Omega}\left(\int_{\Omega} R_\delta(\bx,\by)p_\delta(\by)\mathd\by\right)^2\mathd\bx
\nonumber\\
\le& C\|{ p_\delta}\|_{L^2(\Omega)}^2.\nonumber
\end{align}
where
%\footnote{\color{red} should $R'$ be $\bar{R}'$ below? Maybe we can explain the places where  the relation between $R$ and $\bar{R}$ is used crucially.}
$$R'_\delta(\bx,\by)=C_\delta R'\left(\frac{|\bx-\by|^2}{4\delta^2}\right),\quad R'(r)=\frac{\mathd }{\mathd r}R(r).$$
% Using the similar argument in \eqref{eq:u-error-h1-2}, we have
% \begin{align}
%   \label{eq:u-h1-2}
%   &\|\nabla\left(\frac{1}{2w_\delta(\bx)}\int_\Omega R_\delta(\bx,\by)(\bx-\by)p_\delta(\by)\mathd\by\right)\|_{L^2(\Omega_\delta)}^2
% \le C\|p\|_{L^2(\Omega)}^2
% \end{align}
In addition, direct calculation
%\footnote{\textcolor{red}{Previously it seems that the norm on the RHS below was given as $\|\bff\|_{H^{-1}(\Omega)}$, which is not true?}}
gives that
\begin{align}
  \label{eq:u-h1-3}
  \|\nabla\left(\frac{\delta^2}{w_\delta(\bx)}
  { \int_{\Omega}} \bar{R}_\delta(\bx,\by)\bff(\by)\mathd\by\right)\|_{L^2(\Omega_\delta)}\le C {\|\bff\|_{H^{-1}(\Omega)}}
\end{align}
{
For any $v\in L^2(\Omega_\delta)$,
\begin{align*}
    &\int_{\Omega_\delta} v(\bx)\nabla_{\bx}\left(\frac{\delta^2}{w_\delta(\bx)}\int_{\Omega} \bar{R}_\delta(\bx,\by)\bff(\by)\mathd\by\right)\mathd \bx\\
    =&\int_{\Omega} \left(\int_{\Omega_\delta} v(\bx)\nabla_{\bx}\left(\frac{\delta^2}{w_\delta(\bx)} \bar{R}_\delta(\bx,\by)\right)\mathd \bx\right) \bff(\by)\mathd\by\\
    \le & \|\bff\|_{H^{-1}(\Omega)} \left\|\tilde{v}\right\|_{H^{1}(\Omega_\delta)}
\end{align*}
where 
$$\tilde{v}(\by)=\int_{\Omega_\delta} v(\bx)\nabla_{\bx}\left(\frac{\delta^2}{w_\delta(\bx)} \bar{R}_\delta(\bx,\by)\right)\mathd \bx=\int_{\Omega_\delta} v(\bx)\frac{\delta^2}{w_\delta(\bx)} \nabla_{\bx}\bar{R}_\delta(\bx,\by)\mathd \bx.$$
Here we use the fact that $\omega_\delta(\bx)$ is a constant over $\Omega_\delta$.

Using Lemma \ref{lem:tech}, it is easy to check that 
%\footnote{\color{red} maybe it will be good to put these relations on the kernels and their derivatives in a technical lemma, so we can refer to them not only here but also other places.} 
$$\left\|\tilde{v}\right\|_{H^{1}(\Omega_\delta)}\le C\|v\|_{L^2(\Omega_\delta)}.$$
Then \eqref{eq:u-h1-3} is obtained. 
}

Putting \eqref{eq:u-l2}-\eqref{eq:u-h1-1}-\eqref{eq:u-h1-2}-\eqref{eq:u-h1-3} together, we obtain
\begin{align}
  \label{eq:u-h1}
  \|\bfu_\delta\|_{H^1(\Omega_\delta)}\le C{\|\bff\|_{H^{-1}(\Omega)}}+C\|p_\delta\|_{L^2(\Omega)}
\end{align}

Next, we turn to estimate the pressure $p$.
First, considering the problem
\begin{align}
\label{eq:divergence}
  \nabla \cdot \bm{v}(\bx)=p_\delta(\bx),\quad \bx\in \Omega_\delta.
\end{align}
It is well known (e.g. Section 3.3 of \cite{NS-book}) that if $\Omega_\delta$ satisfies cone condition,
there exists at least one solution of \eqref{eq:divergence}, denoted by $\bm{v}$, such that
%\footnote{\color{red} Do we need the constant $c$ to be independent of $\delta$? Perhaps we have to show this. {\color{blue} The proof is same as that in Appendix B.}}
\begin{align}
\label{eq:bound-v}
  \bm{v}\in H_0^{1}(\Omega_\delta),\quad \|\bm{v}\|_{H^{1}(\Omega_\delta)}\le c\|p_\delta\|_{L^2(\Omega_\delta)}
\end{align}
with $c>0$ independent on $\delta$. Proof of \eqref{eq:bound-v} can be found in Appendix B. 

Then, we extend $\bm{v}$ to $\Omega$ by assigning the value on $\mathcal{V}_\delta$ to be 0 and denote the new function also by $\bm{v}$. Obviously, we have
\begin{align}
\label{eq:bound-v-ext}
  \bm{v}\in H_0^{1}(\Omega_\delta)\cap H_0^{1}(\Omega) ,\quad \|\bm{v}\|_{H^{1}(\Omega)}\le c\|p_\delta\|_{L^2(\Omega_\delta)}.
\end{align}

On the other hand, using the second equation of \eqref{eq:stokes-nonlocal}, $\forall \bx\in \Omega$
\begin{align}
  \label{eq:p-error-0-wp}
\bar{w}_\delta(\bx)p_\delta(\bx)&=\int_\Omega \bar{R}_\delta(\bx,\by)p_\delta(\by)\mathd\by+\frac{1}{2\delta^2}\int_\Omega R_\delta(\bx,\by)(\bx-\by)\cdot \bm{u}_\delta(\by)\mathd \by\nonumber\\
=&\int_{\Omega_\delta} \bar{R}_\delta(\bx,\by)\nabla\cdot \bm{v}(\by)\mathd\by+\frac{1}{2\delta^2}\int_{\Omega_\delta} R_\delta(\bx,\by)(\bx-\by)\cdot \bfu_\delta(\by)\mathd \by\nonumber\\
&+
\frac{1}{2\delta^2}\int_{\V_\delta} R_\delta(\bx,\by)(\bx-\by)\cdot \bfu_\delta(\by)\mathd \by
+\int_{\mathcal{V}_\delta} \bar{R}_\delta(\bx,\by)p_\delta(\by)\mathd\by\nonumber\\
=&-\frac{1}{2\delta^2}\int_{\Omega_\delta} R_\delta(\bx,\by)(\bx-\by)\cdot \bar{\bm{v}}(\by)\mathd\by
+\int_{\mathcal{V}_\delta} \bar{R}_\delta(\bx,\by)p_\delta(\by)\mathd\by
\end{align}
where $\bar{w}_\delta(\bx)=\int_\Omega \bar{R}_\delta(\bx,\by)\mathd\by$ and $\bar{\bm{v}}=\bm{v}-\bfu_\delta$.

Then, it follows that
\begin{align}
\label{eq:p-error-1-0-wp}
&\frac{1}{2\delta^2}\int_{\Omega_\delta}\bar{\bm{v}}(\bx)\left(\int_\Omega R_\delta(\bx,\by)(\bx-\by)p_\delta(\by)\mathd \by\right)\mathd \bx\nonumber\\
=&-\frac{1}{2\delta^2}\int_{\Omega}p_\delta(\bx)\left(\int_{\Omega_\delta} R_\delta(\bx,\by)(\bx-\by)\bar{\bm{v}}(\by)\mathd \by\right)\mathd \bx\nonumber\\
=&\int_{\Omega}p^2_\delta(\bx)\bar{w}_\delta(\bx)\mathd\bx-\int_{\Omega}p_\delta(\bx) \left(\int_{\mathcal{V}_\delta}\bar{R}_\delta(\bx,\by)p_\delta(\by)\mathd\by\right)\mathd\bx.
\end{align}
The first term is positive, thus a good term. The second term becomes
\begin{align}
\label{eq:p-error-1-1-wp}
&  -\int_{\Omega}p_\delta(\bx) \left(\int_{\mathcal{V}_\delta}\bar{R}_\delta(\bx,\by)p_\delta(\by)\mathd\by\right)\mathd\bx\\
=&\int_{\Omega}p_\delta(\bx) \left(\int_{\mathcal{V}_\delta}\bar{R}_\delta(\bx,\by)(p_\delta(\bx)-p_\delta(\by))\mathd\by\right)\mathd\bx
-\int_{\Omega}p^2_\delta(\bx) \left(\int_{\mathcal{V}_\delta}\bar{R}_\delta(\bx,\by)\mathd\by\right)\mathd\bx.\nonumber
\end{align}
The second term of \eqref{eq:p-error-1-1-wp} can be controlled by the first term of \eqref{eq:p-error-1-0-wp}. And the first term is bounded by
\begin{align}
\label{eq:p-error-1-2-wp}
 & \int_{\Omega}p_\delta(\bx) \left(\int_{\mathcal{V}_\delta}\bar{R}_\delta(\bx,\by)(p_\delta(\bx)-p_\delta(\by))\mathd\by\right)\mathd\bx\\
=&\frac{1}{2}\int_{\mathcal{V}_\delta}\int_{\mathcal{V}_\delta}\bar{R}_\delta(\bx,\by)(p_\delta(\bx)-p_\delta(\by))^2\mathd\by\mathd\bx+
 \int_{\Omega_\delta}p_\delta(\bx) \left(\int_{\mathcal{V}_\delta}\bar{R}_\delta(\bx,\by)(p_\delta(\bx)-p_\delta(\by))\mathd\by\right)\mathd\bx\nonumber\\
\ge &\frac{1}{2}\int_{\mathcal{V}_\delta}\int_{\mathcal{V}_\delta}\bar{R}_\delta(\bx,\by)(p_\delta(\bx)-p_\delta(\by))^2\mathd\by\mathd\bx
+\int_{\Omega_\delta} \left(\int_{\mathcal{V}_\delta}\bar{R}_\delta(\bx,\by)(p_\delta(\bx)-p_\delta(\by))^2\mathd\by\right)\mathd\bx\nonumber\\
&-\left|\int_{\Omega_\delta} \left(\int_{\mathcal{V}_\delta}\bar{R}_\delta(\bx,\by)(p_\delta(\bx)-p_\delta(\by))p_\delta(\by)\mathd\by\right)\mathd\bx\right|\nonumber\\
%\ge& \frac{1}{4}\int_{\Omega_\delta} \left(\int_{\mathcal{V}_\delta}\bar{R}_\delta(\bx,\by)p^2(\by)\mathd\by\right)\mathd\bx
\ge&\frac{1}{2}\int_{\Omega} \left(\int_{\mathcal{V}_\delta}\bar{R}_\delta(\bx,\by)(p_\delta(\bx)-p_\delta(\by))^2\mathd\by\right)\mathd\bx
-\frac{1}{2}\int_{\mathcal{V}_\delta} p^2_\delta(\bx)\left(\int_{\Omega_\delta}\bar{R}_\delta(\bx,\by)\mathd\by\right)\mathd\bx.\nonumber
\end{align}

Combining \eqref{eq:p-error-1-0-wp}-\eqref{eq:p-error-1-2-wp}, we get
\begin{align}
\label{eq:p-4}
&\frac{1}{2\delta^2}\int_{\Omega_\delta}\bar{\bm{v}}(\bx)\left(\int_\Omega R_\delta(\bx,\by)(\bx-\by)p_\delta(\by)\mathd \by\right)\mathd \bx
\ge&\int_{\Omega_\delta}p^2_\delta(\bx)\left(\int_{\Omega_\delta}\bar{R}_\delta(\bx,\by)\mathd\by\right)\mathd\bx
\end{align}

% \begin{align}
% \label{eq:p-4}
% &\frac{1}{2\delta^2}\int_{\Omega_\delta}\bar{\bm{v}}(\bx)\left(\int_\Omega R_\delta(\bx,\by)(\bx-\by)p_\delta(\by)\mathd \by\right)\mathd \bx\nonumber\\
% \ge&
% % \int_{\Omega}p^2(\bx)\left(\int_{\Omega_\delta}\bar{R}_\delta(\bx,\by)\mathd\by\right)\mathd\bx
% % -\frac{1}{2}\int_{\mathcal{V}_\delta} p^2(\bx)\left(\int_{\Omega_\delta}\bar{R}_\delta(\bx,\by)\mathd\by\right)\mathd\bx\nonumber\\
% % &+\frac{1}{2}\int_{\Omega} \left(\int_{\mathcal{V}_\delta}\bar{R}_\delta(\bx,\by)(p_\delta(\bx)-p_\delta(\by))^2\mathd\by\right)\mathd\bx\nonumber\\
% % =&
%    \int_{\Omega_\delta}p^2(\bx)\left(\int_{\Omega_\delta}\bar{R}_\delta(\bx,\by)\mathd\by\right)\mathd\bx
% +\frac{1}{2}\int_{\mathcal{V}_\delta} p^2(\bx)\left(\int_{\Omega_\delta}\bar{R}_\delta(\bx,\by)\mathd\by\right)\mathd\bx\nonumber\\
% &
% +\frac{1}{2}\int_{\Omega} \left(\int_{\mathcal{V}_\delta}\bar{R}_\delta(\bx,\by)(p_\delta(\bx)-p_\delta(\by))^2\mathd\by\right)\mathd\bx.
% \end{align}
Now, we are ready to get the estimate of $p_\delta$. 
Multiplying $\bar{\bm{v}}$ on both sides of the first equation of \eqref{eq:stokes-nonlocal} and integrating over $\Omega_\delta$, using the fact that
$\bar{\bm{v}}(\bx)=0,\;\;\bx\in \mathcal{V}_\delta$, we have
\begin{align}
\label{eq:p-0}
   -&\frac{1}{2\delta^2}\int_\Omega\int_\Omega R_\delta(\bx,\by)(\bfu_\delta(\bx)-\bfu_\delta(\by))\cdot (\bar{\bm{v}}(\bx)-\bar{\bm{v}}(\by))\mathd \bx\mathd \by\\
+&\frac{1}{2\delta^2}\int_{\Omega_\delta}\bar{\bm{v}}(\bx)\left(\int_\Omega R_\delta(\bx,\by)(\bx-\by)p_\delta(\by)\mathd \by\right)\mathd \bx=\int_{\Omega_\delta}\bar{\bm{v}}(\bx)\left(\int_{\Omega_\delta} \bar{R}_\delta(\bx,\by)\bff(\by)\mathd \by\right)\mathd \bx.\nonumber
\end{align}

Using \eqref{eq:u-0}, \eqref{eq:u-h1}, \eqref{eq:bound-v-ext}, \eqref{eq:p-0} and \eqref{eq:p-4}, we have
\begin{align}
&\frac{1}{2}\|p\|_{L^2(\Omega_\delta)}^2\nonumber\\
\le&\left( \frac{1}{2\delta^2}\int_\Omega\int_\Omega R_\delta(\bx,\by)|\bfu_\delta(\bx)-\bfu_\delta(\by)|^2\mathd \bx\mathd \by\right)^{1/2}
\left(\frac{1}{2\delta^2}\int_\Omega\int_\Omega R_\delta(\bx,\by)|\bar{\bm{v}}(\bx)-\bar{\bm{v}}(\by)|^2\mathd \bx\mathd \by\right)^{1/2}\nonumber\\
&
+\|\bar{\bm{v}}\|_{H^1(\Omega_\delta)}\|\bm{f}\|_{H^{-1}(\Omega)}\nonumber\\
\le & \left( \frac{1}{2\delta^2}\int_\Omega\int_\Omega R_\delta(\bx,\by)|\bfu_\delta(\bx)-\bfu_\delta(\by)|^2\mathd \bx\mathd \by\right)^{1/2}
\left(\left(\frac{1}{2\delta^2}\int_\Omega\int_\Omega R_\delta(\bx,\by)|\bm{v}(\bx)-\bm{v}(\by)|^2\mathd \bx\mathd \by\right)^{1/2}\right.\nonumber\\
&\left.+
\left(\frac{1}{2\delta^2}\int_\Omega\int_\Omega R_\delta(\bx,\by)|\bm{u}_\delta(\bx)-\bm{u}_\delta(\by)|^2\mathd \bx\mathd \by\right)^{1/2}\right)
+(\|\bm{v}\|_{H^1(\Omega_\delta)}+\|\bm{u}_\delta\|_{H^1(\Omega_\delta)})\|\bm{f}\|_{H^{-1}(\Omega)} \nonumber\\
\le &\|\bm{u}_\delta\|_{H^1(\Omega_\delta)}\|\bm{f}\|_{H^{-1}(\Omega)}+\|\bm{u}_\delta\|_{H^1(\Omega_\delta)}^{1/2}\|\bm{f}\|_{H^{-1}(\Omega)}^{1/2}\|\bm{v}\|_{H^1(\Omega_\delta)}+C(\|p_\delta\|_{L^2(\Omega_\delta)}
+\|\bff\|_{H^{-1}(\Omega)})\|\bm{f}\|_{H^{-1}(\Omega)} \nonumber\\
\le& C(\|p_\delta\|_{L^2(\Omega_\delta)}
+\|\bff\|_{H^{-1}(\Omega)})\|\bm{f}\|_{H^{-1}(\Omega)}.\nonumber\\
&
\label{eq:p-inner}
\end{align}

Using \eqref{eq:u-0}, \eqref{eq:u-h1}, \eqref{eq:p-inner} and Lemma \ref{cor:l2-inner-ave}, we have
\begin{align}
  \|p_\delta\|_{L^2(\Omega)}^2\le& C\|\bfu_\delta\|_{L^2(\Omega)}\|\bff\|_{H^{-1}(\Omega)}+
C\|p_\delta\|_{L^2(\Omega_\delta)}^2\nonumber\\
\le& C(\|p_\delta\|_{L^2(\Omega)}
+\|\bff\|_{H^{-1}(\Omega)})\|\bm{f}\|_{H^{-1}(\Omega)}\nonumber
\end{align}
%Together with \eqref{eq:p-inner}, we have
which implies 
\begin{align}
\label{eq:p-global}
\|p_\delta\|_{L^2(\Omega)}\le C\|\bm{f}\|_{H^{-1}(\Omega)}.
\end{align}
This also gives the $H^1$ estimate of $\bfu_\delta$ using \eqref{eq:u-h1},
\begin{align}
\label{eq:u-h1-final}
\|\bfu_\delta\|_{H^1(\Omega_\delta)}\le C\|\bm{f}\|_{H^{-1}(\Omega)}.
\end{align}
and using \eqref{eq:p-pbar},
\begin{align}
  \label{eq:p-global-final}
\|\bar{p}_\delta\|_{L^2(\Omega)}\le \|p_\delta\|_{L^2(\Omega)}+\frac{1}{|\Omega|}\left|\int_\Omega p_\delta(\bx)\mathd\bx\right|\le  C\|\bm{f}\|_{H^{-1}(\Omega)}.
\end{align}
Note that in the above,  the fact that 
\begin{align*}
  \frac{1}{|\Omega|}\left|\int_\Omega p_\delta(\bx)\mathd\bx\right|\le \frac{1}{\sqrt{|\Omega|}}\|p_\delta\|_{L^2(\Omega)},
\end{align*}
is used.
\end{proof}

%%% Local Variables:
%%% mode: latex
%%% TeX-master: t
%%% End:

%This theorem can be proved following the similar idea as that in the proof of Theorem \ref{thm:converge-integral}. The details of the proof can be found in Appendix B.

\section{Vanishing nonlocality}
\label{sec:vanish}

Besides the well-posedness,  we are  also interested in the limiting behavior of the nonlocal Stokes system \eqref{eq:stokes-nonlocal} as the nonlocality vanishes, 
i.e. $\delta\rightarrow 0$.
In this section, under some assumptions, we prove that solutions of the nonlocal Stokes system converge to the solution of the
Stokes system as $\delta\rightarrow 0$. Furthermore, we give an estimate on the convergence rate. The result is summarized in Theorem \ref{thm:converge-integral}.

Before stating the main theorem, we give several technical results that are used to prove the main theorem.

% \begin{lemma}(\cite{SS14})
% \label{lem:coercivity}
%  For any function $u\in L^2(\Omega)$,
% there exists a constant $C>0$ only depends on $\M$, such that
%   \begin{eqnarray}
%     \frac{1}{\delta^2}\int_{\Omega}\int_{\Omega} R_\delta(\bx,\by)(u(\bx)-u(\by))^2\mathd\bx\mathd\by \ge C\int_\Omega |\nabla v|^2\mathd \bx,\nonumber
%   \end{eqnarray}
% where
% \begin{eqnarray}
% v(\bx)=\frac{1}{w_\delta(\bx)}\int_{\Omega}R_\delta(\bx,\by)u(\by)\mathd \by, \nonumber
% \end{eqnarray}
% and $\D w_\delta(\bx) = \int_{\Omega}R_\delta\left(\bx,\by\right)\mathd \by$.
% \end{lemma}

We also need the following theorem on the order of the nonlocal approximation
{which can be proved via simple Taylor expansion.}
%whose proof can be found in \cite{SS14}.
\begin{theorem}% (\cite{SS14} Theorem 5.1)
Let 
 \begin{eqnarray*}
   r(\bx)=-\frac{1}{\delta^2}\int_{\M}R_\delta(\bx,\by)(u(\bx)-u(\by))\mathd\by-\int_{\M}\bar{R}_\delta(\bx,\by)\Delta u(\by)\mathd\by,\quad \forall \bx\in \Omega_\delta.
 \end{eqnarray*}
There exist constants
$C, T_0$ depending only on $\M$, so that for any $\delta\le T_0$,
 {for $u\in H^3(\M)$,}
\begin{eqnarray}
\label{eq:integral_error_l2}
\left\|r(\bx)\right\|_{L^2(\M_\delta)}&\le& C{\delta\|u\|_{H^3(\Omega)}},\\
\label{eq:integral_error_h1}
\left\|\nabla r(\bx)\right\|_{L^2(\M_\delta)}
&\leq& C \|u\|_{H^3(\Omega)}.
%\|u\|_{H^3(\Omega)}.
\end{eqnarray}
\label{thm:integral_error}
\end{theorem}

We then have the main result of this section regarding the convergence of the nonlocal Stokes system as the nonlocality vanishes.
\begin{theorem}
  \label{thm:converge-integral}
Let $\bfu(\bx)$, $p(\bx)$ be solution of Stokes system \eqref{eq:stokes} and $\bfu_\delta(\bx),\; p_\delta(\bx)$ be solution of nonlocal Stokes system \eqref{eq:stokes-nonlocal}
with $\bff\in H^1(\Omega)$.
There exists a constant $C>0$ that only depends on $\M$ and $R$, such that
\begin{eqnarray*}
  \|\bfu-\bfu_\delta\|_{H^1(\M_\delta)}+\|p-p_\delta\|_{L^2(\Omega)}\le C\sqrt{\delta}\|f\|_{H^1(\M)}
\end{eqnarray*}
\end{theorem}
\begin{proof}
Let $\be_\delta(\bx)=\bfu(\bx)-\bfu_\delta(\bx)$ and $d_\delta=p-p_\delta-\frac{1}{|\Omega_\delta|}\int_{\Omega_\delta}(p(\bx)-p_\delta(\bx))\mathd\bx$, then $\be_\delta$ and $d_\delta$ satisfy
\begin{align}
  \label{eq:errors}
\left\{\begin{array}{rclc}
\D  -\frac{1}{\delta^2}\int_\Omega R_\delta(\bx,\by)(\be_\delta(\bx)-\be_\delta(\by))\mathd \by
+\D \frac{1}{2\delta^2}\int_\Omega R_\delta(\bx,\by)(\bx-\by) d_\delta(\by)\mathd \by&=&\D \bm{r}_\bfu(\bx),& \bx\in \Omega_\delta, \\
\be_\delta(\bx)&=&\bfu(\bx),&\bx\in \mathcal{V}_\delta,\\
\D \frac{1}{2\delta^2}\int_\Omega R_\delta(\bx,\by)(\bx-\by)\cdot \be_\delta(\by)\mathd \by - \int_\Omega \bar{R}_\delta(\bx,\by)(d_\delta(\bx)-d_\delta(\by))\mathd \by & = &
\D r_p(\bx), & \bx\in \Omega,\\
\D \int_{\Omega_\delta}d_\delta(\bx)\mathd\bx&=&0,&
\end{array}\right.
\end{align}
where
\begin{align}
  \label{eq:residual-u}
  \bm{r}_\bfu(\bx)=&\int_{\M}\bar{R}_\delta(\bx,\by)\Delta \bfu(\by)\mathd\by+\frac{1}{\delta^2}\int_{\M}R_\delta(\bx,\by)(\bfu(\bx)-\bfu(\by))\mathd\by,&\quad \forall \bx\in \Omega_\delta\\
\label{eq:residual-p}
 r_p(\bx)= & - \int_\Omega \bar{R}_\delta(\bx,\by)(p(\bx)-p(\by))\mathd \by,& \quad \forall \bx\in \Omega.
\end{align}
First, we focus on the following estimate
\begin{align}
  \label{eq:error-1}
%&\int_{\Omega_\delta}\bm{r}_\bfu(\bx)\cdot \be_\delta(\bx)\mathd\bx\nonumber\\
&  \frac{1}{\delta^2}\int_{\M_\delta}\be_\delta(\bx)\cdot\int_{\M}R_\delta(\bx,\by)(\be_\delta(\bx)-\be_\delta(\by))\mathd\by\mathd\bx\\
=& \frac{1}{\delta^2}\int_{\M_\delta}\be_\delta(\bx)\cdot\int_{\M_\delta}R_\delta(\bx,\by)(\be_\delta(\bx)-\be_\delta(\by))\mathd\by\mathd\bx+
 \frac{1}{\delta^2}\int_{\M_\delta}\be_\delta(\bx)\cdot\int_{\V_\delta}R_\delta(\bx,\by)(\be_\delta(\bx)-\be_\delta(\by))\mathd\by\mathd\bx\nonumber\\
=& \frac{1}{2\delta^2}\int_{\M_\delta}\int_{\M_\delta}R_\delta(\bx,\by)|\be_\delta(\bx)-\be_\delta(\by)|^2\mathd\bx\mathd\by
+\frac{1}{\delta^2}\int_{\M_\delta}\be_\delta(\bx)\cdot\int_{\V_\delta}R_\delta(\bx,\by)(\be_\delta(\bx)-\be_\delta(\by))\mathd\by\mathd\bx.\nonumber
\end{align}
The second term of the right hand side of \eqref{eq:error-1} can be calculated as
\begin{align}
  \label{eq:error-2}
&  \frac{1}{\delta^2}\int_{\M_\delta}\be_\delta(\bx)\cdot\int_{\V_\delta}R_\delta(\bx,\by)(\be_\delta(\bx)-\be_\delta(\by))\mathd\by\mathd\bx\\
=&\frac{1}{\delta^2}\int_{\M_\delta}|\be_\delta(\bx)|^2\left(\int_{\V_\delta}R_\delta(\bx,\by)\mathd\by\right)\mathd\bx
-\frac{1}{\delta^2}\int_{\M_\delta}\be_\delta(\bx)\cdot\left(\int_{\V_\delta}R_\delta(\bx,\by)\bfu(\by)\mathd\by\right)\mathd\bx.
\nonumber
\end{align}
Here we use the definition of $\be_\delta$ and the volume constraint condition $\bfu_\delta(\bx)=0,\;\bx\in \V_\delta$ to get that $\be_\delta(\bx)=u(\bx),\;\bx\in \V_\delta$. %\footnote{\color{red} Note that it might be simpler if one moves the term involving
%$\bfu$ to the right hand side of the weak form, then recognizing all the other 
%remaining terms on the left hand side are part of $a([\be_\delta, d_\delta],[\be_\delta, d_\delta])$ so 
%one can use the coercity already established before to get much simpler proof. Much of the derivation here are proving again similar coercivity estimates. But it's OK to keep the current proof if too much changes have to be done.
%}.

The first term is positive which is good for us. We only need to bound the second term of \eqref{eq:error-2}.
%to show that it can be controlled by the first term. 
First, the second term can be bounded as following
\begin{align}
  \label{eq:error-3}\quad\quad
&  \frac{1}{\delta^2}\left|\int_{\M_\delta}\be_\delta(\bx)\cdot\left(\int_{\V_\delta}R_\delta(\bx,\by)\bfu(\by)\mathd\by\right)\mathd\bx\right|\\
\le& \frac{1}{\delta^2}\int_{\M_\delta}|\be_\delta(\bx)|\left(\int_{\V_\delta}R_\delta(\bx,\by)\mathd\by\right)^{1/2}\left(\int_{\V_\delta}R_\delta(\bx,\by)|\bfu(\by)|^2\mathd\by\right)^{1/2}\mathd\bx\nonumber\\
\le&\frac{1}{\delta^2}\left(\int_{\M_\delta}\frac{1}{2}|\be_\delta(\bx)|^2\left(\int_{\V_\delta}R_\delta(\bx,\by)\mathd\by\right)\mathd\bx+
2\int_{\M_\delta}\left(\int_{\V_\delta}R_\delta(\bx,\by)|\bfu(\by)|^2\mathd\by\right)\mathd\bx\right)\nonumber\\
\le& \frac{1}{2\delta^2}\int_{\M_\delta}|\be_\delta(\bx)|^2\left(\int_{\V_\delta}R_\delta(\bx,\by)\mathd\by\right)\mathd\bx+
\frac{2}{\delta^2}\int_{\V_\delta}|\bfu(\by)|^2\left(\int_{\M_\delta} R_\delta(\bx,\by)\mathd \bx\right)\mathd\by\nonumber\\
\le& \frac{1}{2\delta^2}\int_{\M_\delta}|\be_\delta(\bx)|^2\left(\int_{\V_\delta}R_\delta(\bx,\by)\mathd\by\right)\mathd\bx+
\frac{C}{\delta^2}\int_{\V_\delta}|\bfu(\by)|^2\mathd\by\nonumber\\
\le& \frac{1}{2\delta^2}\int_{\M_\delta}|\be_\delta(\bx)|^2\left(\int_{\V_\delta}R_\delta(\bx,\by)\mathd\by\right)\mathd\bx+
C\delta\|\bff\|^2_{H^1(\M)}.\nonumber
\end{align}
Here we use Lemma \ref{lem:u-boundary} in Appendix \ref{sec:appendix3} to get the last inequality.
By substituting \eqref{eq:error-3}, \eqref{eq:error-2} in \eqref{eq:error-1}, we get
\begin{eqnarray}
  \label{eq:error-est}
&& \left| \frac{1}{\delta^2}\int_{\M_\delta}\be_\delta(\bx)\cdot\int_{\M}R_\delta(\bx,\by)(\be_\delta(\bx)-\be_\delta(\by))\mathd\by\mathd\bx\right|\\
&\ge&  \frac{1}{2\delta^2}\int_{\M_\delta}\int_{\M_\delta}R_\delta(\bx,\by)|\be_\delta(\bx)-\be_\delta(\by)|^2\mathd\bx\mathd\by\nonumber\\
&&+
\frac{1}{2\delta^2}\int_{\M_\delta}|\be_\delta(\bx)|^2\left(\int_{\V_\delta}R_\delta(\bx,\by)\mathd\by\right)\mathd\bx-C\|\bff\|_{H^1(\M)}^2\delta\nonumber.
\end{eqnarray}
This is the key estimate to show the convergence.

We also need the following bound
\begin{align}
  \label{eq:u-up-1}
  &\left|\frac{1}{\delta^2}\int_{\Omega_\delta}\be_\delta(\bx)\cdot\left(\int_\Omega R_\delta(\bx,\by)(\bx-\by) d_\delta(\by)\mathd\by\right)\mathd\bx
+\frac{1}{\delta^2}\int_{\Omega}d_\delta(\bx)\left(\int_\Omega R_\delta(\bx,\by)(\bx-\by)\cdot \be_\delta(\by)\mathd\by\right)\mathd\bx\right|\\
= & \left|\frac{1}{\delta^2}\int_{\Omega}d_\delta(\bx)\left(\int_{\V_\delta} R_\delta(\bx,\by)(\bx-\by)\cdot \be_\delta(\by)\mathd\by\right)\mathd\bx\right|\nonumber\\
\le& \frac{1}{\delta}\int_{\Omega}\left(\int_{\V_\delta} R_\delta(\bx,\by)|d_\delta(\bx)||\bfu(\by)|\mathd\by\right)\mathd\bx\nonumber\\
\le&\frac{1}{\delta}\left[\int_{\Omega}\left(\int_{\V_\delta} R_\delta(\bx,\by)|d_\delta(\bx)|^2\mathd\by\right)\mathd\bx
\int_{\Omega}\left(\int_{\V_\delta} R_\delta(\bx,\by)|\bfu(\by)|^2\mathd\by\right)\mathd\bx\right]^{1/2}\nonumber\\
\le&C\sqrt{\delta}\|\bff\|_{H^1(\Omega)}\|d_\delta\|_{L^2(\Omega)}.\nonumber
%\left[\int_{\Omega}|d_\delta(\bx)|^2\left(\int_{\V_\delta} R_\delta(\bx,\by)\mathd\by\right)\mathd\bx\right]^{1/2}
\end{align}
Multiplying $\be_\delta(\bx)$, $d_\delta(\bx)$ on both sides of the first and third equations in \eqref{eq:errors} and integrating over $\Omega_\delta$, $\Omega$
respectively and adding them together, using \eqref{eq:error-est}, \eqref{eq:u-up-1}, we have

\begin{align}
  \label{eq:u-error}
 & \frac{1}{\delta^2}\int_{\Omega_\delta}\int_{\Omega_\delta}R_\delta(\bx,\by)|\be_\delta(\bx)-\be_\delta(\by)|^2\mathd\bx\mathd\by+\frac{1}{2\delta^2}\int_{\Omega_\delta}|\be_\delta(\bx)|^2
\left(\int_{\V_\delta}R_\delta(\bx,\by)\mathd\by\right)\mathd\bx\\
&+\int_{\Omega}\int_{\Omega}R_\delta(\bx,\by)|d_\delta(\bx)-d_\delta(\by)|^2\mathd\bx\mathd\by\nonumber\\
\le& (\|\bm{r}_\bfu\|_{L^2(\Omega_\delta)})\|\bm{e}_\delta\|_{L^2(\Omega_\delta)}+\|r_p\|_{L^2(\Omega)}\|d_\delta\|_{L^2(\Omega)}+C\sqrt{\delta}\|\bff\|_{H^1(\Omega)}\|d_\delta\|_{L^2(\Omega)}
+C\delta\|\bff\|_{H^1(\Omega)}^2.\nonumber
\end{align}
To simplify the notation, we denote the right hand side of \eqref{eq:u-error} as $Q^2$.

It is well known (e.g. Section 3.3 of \cite{NS-book}) that with the condition that $$\int_{\Omega_\delta}d_\delta(\bx)\mathd\bx=0,$$
there exists at least one function $\bm{\psi}\in H_0^{1}(\Omega_\delta)
$, such that
%\footnote{\color{red} Similar to discussion before, do we need to prove that the constant $c$ is independent of $\delta$?.}
\begin{align} 
\nabla \cdot \bm{\psi}(\bx)=d_\delta(\bx),\quad \bx\in \Omega_\delta, \quad \text{and}
\quad \|\bm{\psi}\|_{H^{1}(\Omega_\delta)}\le c\|d_\delta\|_{L^2(\Omega_\delta)}.
\label{eq:bound-psi}
\end{align}
and $c$ is a constant independent on $\delta$, the proof can be found in Appendix  \ref{sec:appendix4}. 

Then, we extend $\bm{\psi}$ to $\Omega$ by assigning the value on $\mathcal{V}_\delta$ to be 0 and denote the new function also by $\bm{\psi}$. Obviously, we have
\begin{align}
\label{eq:bound-psi-ext-error}
  \bm{\psi}\in H_0^{1}(\Omega_\delta)\cap H_0^{1}(\Omega) ,\quad \|\bm{\psi}\|_{H^{1}(\Omega)}\le c\|d_\delta\|_{L^2(\Omega_\delta)}
\end{align}
Using the third equation of \eqref{eq:errors}, we have
\begin{align}
  \label{eq:p-error-0}
\bar{w}_\delta(\bx)d_\delta(\bx)&=\int_\Omega \bar{R}_\delta(\bx,\by)d_\delta(\by)\mathd\by+\frac{1}{2\delta^2}\int_\Omega R_\delta(\bx,\by)(\bx-\by)\cdot \be_\delta(\by)\mathd \by-r_p(\bx)\nonumber\\
=&\int_{\Omega_\delta} \bar{R}_\delta(\bx,\by)\nabla\cdot \bm{\psi}(\by)\mathd\by+\frac{1}{2\delta^2}\int_{\Omega_\delta} R_\delta(\bx,\by)(\bx-\by)\cdot \be_\delta(\by)\mathd \by\nonumber\\
&+
\frac{1}{2\delta^2}\int_{\V_\delta} R_\delta(\bx,\by)(\bx-\by)\cdot \be_\delta(\by)\mathd \by
+\int_{\mathcal{V}_\delta} \bar{R}_\delta(\bx,\by)d_\delta(\by)\mathd\by-r_p(\bx)\nonumber\\
%=&-\frac{1}{2\delta^2}\int_{\Omega_\delta} R_\delta(\bx,\by)(\bx-\by)\cdot \bm{v}(\by)\mathd\by+\frac{1}{2\delta^2}\int_{\Omega_\delta} R_\delta(\bx,\by)(\bx-\by)\cdot \bfu(\by)\mathd \by
%+\int_{\mathcal{V}_\delta} \bar{R}_\delta(\bx,\by)p(\by)\mathd\by\nonumber\\
=&-\frac{1}{2\delta^2}\int_{\Omega_\delta} R_\delta(\bx,\by)(\bx-\by)\cdot \bar{\bm{\psi}}(\by)\mathd\by
+
\frac{1}{2\delta^2}\int_{\V_\delta} R_\delta(\bx,\by)(\bx-\by)\cdot \bfu(\by)\mathd \by\nonumber\\
&+\int_{\mathcal{V}_\delta} \bar{R}_\delta(\bx,\by)d_\delta(\by)\mathd\by-r_p(\bx)
\end{align}
where $\bar{w}_\delta(\bx)=\int_\Omega \bar{R}_\delta(\bx,\by)\mathd\by$ and $\bar{\bm{\psi}}=\bm{\psi}-\be_\delta$.

Then, it follows that
\begin{align}
\label{eq:p-error-1-0}
&\frac{1}{2\delta^2}\int_{\Omega_\delta}\bar{\bm{\psi}}(\bx)\left(\int_\Omega R_\delta(\bx,\by)(\bx-\by)d_\delta(\by)\mathd \by\right)\mathd \bx\nonumber\\
=&-\frac{1}{2\delta^2}\int_{\Omega}d_\delta(\bx)\left(\int_{\Omega_\delta} R_\delta(\bx,\by)(\bx-\by)\bar{\bm{\psi}}(\by)\mathd \by\right)\mathd \bx\nonumber\\
=&\int_{\Omega}d_\delta^2(\bx)\bar{w}_\delta(\bx)\mathd\bx-\int_{\Omega}d_\delta(\bx) \left(\int_{\mathcal{V}_\delta}\bar{R}_\delta(\bx,\by)d_\delta(\by)\mathd\by\right)\mathd\bx
\nonumber\\
&-\int_{\Omega}d_\delta(\bx) \left(\int_{\mathcal{V}_\delta}R_\delta(\bx,\by)(\bx-\by)\cdot \bfu(\by)\mathd\by\right)\mathd\bx
+\int_\Omega d_\delta(\bx)r_p(\bx)\mathd\bx
\end{align}
The first term is positive which is a good term. The second term becomes
\begin{align}
\label{eq:p-error-1-1}
&  -\int_{\Omega}d_\delta(\bx) \left(\int_{\mathcal{V}_\delta}\bar{R}_\delta(\bx,\by)d_\delta(\by)\mathd\by\right)\mathd\bx\\
=&\int_{\Omega}d_\delta(\bx) \left(\int_{\mathcal{V}_\delta}\bar{R}_\delta(\bx,\by)(d_\delta(\bx)-d_\delta(\by))\mathd\by\right)\mathd\bx
-\frac{1}{2\delta^2}\int_{\Omega}d_\delta^2(\bx) \left(\int_{\mathcal{V}_\delta}\bar{R}_\delta(\bx,\by)\mathd\by\right)\mathd\bx.\nonumber
\end{align}
The second term of \eqref{eq:p-error-1-1} can be controlled by the first term of \eqref{eq:p-error-1-0}. And the first term is bounded by
\begin{align}
\label{eq:p-error-1-2}
 & \int_{\Omega}d_\delta(\bx) \left(\int_{\mathcal{V}_\delta}\bar{R}_\delta(\bx,\by)(d_\delta(\bx)-d_\delta(\by))\mathd\by\right)\mathd\bx\\
=&\frac{1}{2}\int_{\mathcal{V}_\delta}\int_{\mathcal{V}_\delta}\bar{R}_\delta(\bx,\by)(d_\delta(\bx)-d_\delta(\by))^2\mathd\by\mathd\bx+
 \int_{\Omega_\delta}d_\delta(\bx) \left(\int_{\mathcal{V}_\delta}\bar{R}_\delta(\bx,\by)(d_\delta(\bx)-d_\delta(\by))\mathd\by\right)\mathd\bx\nonumber\\
\ge &\frac{1}{2}\int_{\mathcal{V}_\delta}\int_{\mathcal{V}_\delta}\bar{R}_\delta(\bx,\by)(d_\delta(\bx)-d_\delta(\by))^2\mathd\by\mathd\bx
+\int_{\Omega_\delta} \left(\int_{\mathcal{V}_\delta}\bar{R}_\delta(\bx,\by)(d_\delta(\bx)-d_\delta(\by))^2\mathd\by\right)\mathd\bx\nonumber\\
&-\left|\int_{\Omega_\delta} \left(\int_{\mathcal{V}_\delta}\bar{R}_\delta(\bx,\by)(d_\delta(\bx)-d_\delta(\by))d_\delta(\by)\mathd\by\right)\mathd\bx\right|\nonumber\\
%\ge& \frac{1}{4}\int_{\Omega_\delta} \left(\int_{\mathcal{V}_\delta}\bar{R}_\delta(\bx,\by)p^2(\by)\mathd\by\right)\mathd\bx
\ge&\frac{1}{2}\int_{\Omega} \left(\int_{\mathcal{V}_\delta}\bar{R}_\delta(\bx,\by)(d_\delta(\bx)-d_\delta(\by))^2\mathd\by\right)\mathd\bx
-\frac{1}{2}\int_{\mathcal{V}_\delta} d_\delta^2(\bx)\left(\int_{\Omega_\delta}\bar{R}_\delta(\bx,\by)\mathd\by\right)\mathd\bx.\nonumber
\end{align}
Combining \eqref{eq:p-error-1-0}-\eqref{eq:p-error-1-2}, we get
\begin{align}
\label{eq:p-error-1}
&\frac{1}{2\delta^2}\int_{\Omega_\delta}\bar{\bm{\psi}}(\bx)\left(\int_\Omega R_\delta(\bx,\by)(\bx-\by)d_\delta(\by)\mathd \by\right)\mathd \bx\\
\ge&\int_{\Omega_\delta}d_\delta^2(\bx)\left(\int_{\Omega_\delta}\bar{R}_\delta(\bx,\by)\mathd\by\right)\mathd\bx-\frac{1}{2\delta^2}\int_\Omega d_\delta(\bx)
\left(\int_{\V_\delta} R_\delta(\bx,\by)(\bx-\by)\cdot \bfu(\by)\mathd \by\right)\mathd\bx\nonumber\\
&+\int_\Omega d_\delta(\bx)r_p(\bx)\mathd\bx.\nonumber
\end{align}
In addition, we have
\begin{align}
  \label{eq:p-error-2}
  &\left|\frac{1}{2\delta^2}\int_\Omega d_\delta(\bx)
\left(\int_{\V_\delta} R_\delta(\bx,\by)(\bx-\by)\cdot \bfu(\by)\mathd \by\right)\mathd\bx\right|\\
\le&\frac{1}{2\delta}\int_\Omega |d_\delta(\bx)|
\left(\int_{\V_\delta} R_\delta(\bx,\by)|\bfu(\by)|\mathd \by\right)\mathd\bx\nonumber\\
\le&\frac{1}{2\delta}\left[\int_\Omega |d_\delta(\bx)|^2
\left(\int_{\V_\delta} R_\delta(\bx,\by)\mathd \by\right)\mathd\bx
\int_\Omega
\left(\int_{\V_\delta} R_\delta(\bx,\by)|\bfu(\by)|^2\mathd \by\right)\mathd\bx\right]^{1/2}\nonumber\\
\le&C\sqrt{\delta}\|\bff\|_{H^1(\Omega)}\|d_\delta\|_{L^2(\Omega)}.\nonumber
\end{align}
and
%\footnote{\color{red} need to verify somewhere that 
%$\|r_p\|\leq C\delta\|p\|_{H^1(\Omega)}$ so one can get the following estimate?}
\begin{align}
  \label{eq:p-error-3}
  \left|\int_\Omega d_\delta(\bx)r_p(\bx)\mathd\bx\right|=&\left|\int_\Omega d_\delta(\bx)\left(\int_\Omega \bar{R}_\delta(\bx,\by)(p(\bx)-p(\by))\mathd\by\right)\mathd\bx\right|\\
\le& C\delta\|p\|_{H^1(\Omega)}\|d_\delta\|_{L^2(\Omega)}\nonumber\\
\le&C\delta\|\bff\|_{H^1(\Omega)}\|d_\delta\|_{L^2(\Omega)}.\nonumber
\end{align}
% \begin{align}
%   \label{eq:psi}
%   \bm{\psi}(\bx)=\left\{\begin{array}{cc}\bm{\psi}(\bx)-\bm{e}_\delta(\bx),& \bx\in \Omega_\delta,\\
% -\bfu(\bx),& \bx\in \V_\delta.
% \end{array}\right.
% \end{align}

Multiplying $\bar{\bm{\psi}}$ on both sides of the first equation of \eqref{eq:errors} and using \eqref{eq:p-error-1}, \eqref{eq:p-error-2}, \eqref{eq:p-error-3},
we have
\begin{align}
\label{eq:p-error-4}
\|d_\delta\|_{L^2(\Omega_\delta)}^2\le& \frac{1}{\delta^2}\int_{\Omega_\delta}\int_{\Omega_\delta} R_\delta(\bx,\by)(\be_\delta(\bx)-\be_\delta(\by))\cdot(\bar{\bm{\psi}}(\bx)-\bar{\bm{\psi}}(\by))\mathd \bx\mathd \by\nonumber\\
&+\frac{1}{\delta^2}\int_{\Omega_\delta}\bar{\bm{\psi}}(\bx)\cdot\left(\int_{\V_\delta} R_\delta(\bx,\by)(\be_\delta(\bx)-\be_\delta(\by))\mathd \by\right)\mathd \bx\nonumber\\
&
+\|\bar{\bm{\psi}}\|_{L^2(\Omega_\delta)}(\|\bm{r}_\bfu\|_{L^2(\Omega_\delta)})+C\sqrt{\delta}\|\bff\|_{H^1(\Omega)}\|d_\delta\|_{L^2(\Omega)}% \nonumber\\
% \le & \left( \frac{1}{2\delta^2}\int_\Omega\int_\Omega R_\delta(\bx,\by)|\be_\delta(\bx)-\be_\delta(\by)|^2\mathd \bx\mathd \by\right)^{1/2}
% \left(\left(\frac{1}{2\delta^2}\int_\Omega\int_\Omega R_\delta(\bx,\by)|\bm{\psi}(\bx)-\bm{\psi}(\by)|^2\mathd \bx\mathd \by\right)^{1/2}\right.\nonumber\\
% &\left.+
% \left(\frac{1}{2\delta^2}\int_\Omega\int_\Omega R_\delta(\bx,\by)|\be_\delta(\bx)-\be_\delta(\by)|^2\mathd \bx\mathd \by\right)^{1/2}\right)
% +(\|\bm{\psi}\|_{L^2(\Omega_\delta)}+\|\be_\delta\|_{L^2(\Omega_\delta)})\|\bm{r}_\bfu\|_{L^2(\Omega)} \nonumber\\
% &+\left(\frac{1}{\delta^2}\int_{\Omega}\left(\int_{\V_\delta}R_\delta(\bx,\by)|\bfu(\bx)|^2\mathd\bx\right)\mathd\by\right)^{1/2}
% \left(\frac{1}{\delta^2}\int_{\V_\delta}\left(\int_{\Omega} R_\delta(\bx,\by)|\be_\delta(\bx)-\be_\delta(\by)|^2\mathd \by\right)\mathd \bx\right)^{1/2}\nonumber\\
% \le &\|\bm{e}_\delta\|_{L^2(\Omega)}\|\bm{r}_\bfu\|_{L^2(\Omega)}+\|\bm{e}_\delta\|_{L^2(\Omega)}^{1/2}\|\bm{r}_\bfu\|_{L^2(\Omega)}^{1/2}\|\bm{\psi}\|_{H^1(\Omega)}+C\sqrt{\delta}C(\|p\|_{L^2(\Omega_\delta)}
% +\|\bff\|_{L^2(\Omega)})\|\bm{f}\|_{L^2(\Omega)} \nonumber\\
% \le& C(\|p\|_{L^2(\Omega_\delta)}
% +\|\bff\|_{L^2(\Omega)})\|\bm{f}\|_{L^2(\Omega)}.
\end{align}
The first term can be bounded as
\begin{align}
  \label{eq:p-error-5}
&\left|\frac{1}{\delta^2}\int_{\Omega_\delta}\int_{\Omega_\delta} R_\delta(\bx,\by)(\be_\delta(\bx)-\be_\delta(\by))\cdot(\bar{\bm{\psi}}(\bx)-\bar{\bm{\psi}}(\by))\mathd \bx\mathd \by\right|\\
 \le& \left( \frac{1}{\delta^2}\int_{\Omega_\delta}\int_{\Omega_\delta} R_\delta(\bx,\by)|\be_\delta(\bx)-\be_\delta(\by)|^2\mathd \bx\mathd \by\right)^{1/2}
\left(\frac{1}{\delta^2}\int_{\Omega_\delta}\int_{\Omega_\delta} R_\delta(\bx,\by)|\bar{\bm{\psi}}(\bx)-\bar{\bm{\psi}}(\by)|^2\mathd \bx\mathd \by\right)^{1/2}\nonumber\\
\le & \left( \frac{1}{\delta^2}\int_{\Omega_\delta}\int_{\Omega_\delta} R_\delta(\bx,\by)|\be_\delta(\bx)-\be_\delta(\by)|^2\mathd \bx\mathd \by\right)^{1/2}
\left(\left(\frac{1}{\delta^2}\int_{\Omega_\delta}\int_{\Omega_\delta} R_\delta(\bx,\by)|\bm{\psi}(\bx)-\bm{\psi}(\by)|^2\mathd \bx\mathd \by\right)^{1/2}\right.\nonumber\\
&\left.+
\left(\frac{1}{\delta^2}\int_{\Omega_\delta}\int_{\Omega_\delta} R_\delta(\bx,\by)|\be_\delta(\bx)-\be_\delta(\by)|^2\mathd \bx\mathd \by\right)^{1/2}\right)\nonumber\\
\le& Q^2+CQ\|\bm{\psi}\|_{H^1(\Omega_\delta)}\le Q^2+CQ\|d_\delta\|_{L^2(\Omega_\delta)},\nonumber
\end{align}

The estimate of the second term of \eqref{eq:p-error-4} is more involved. First
\begin{align}
  \label{eq:p-error-6}
&  \left|\frac{1}{\delta^2}\int_{\Omega_\delta}\bar{\bm{\psi}}(\bx)\cdot\left(\int_{\V_\delta} R_\delta(\bx,\by)(\be_\delta(\bx)-\be_\delta(\by))\mathd \by\right)\mathd \bx\right|\\
\le& \left|\frac{1}{\delta^2}\int_{\Omega_\delta}\left(\int_{\V_\delta} R_\delta(\bx,\by)(\bar{\bm{\psi}}(\bx)-\bar{\bm{\psi}}(\by))\cdot(\be_\delta(\bx)-\be_\delta(\by))\mathd \by\right)\mathd \bx\right|
\nonumber\\
&+\left|\frac{1}{\delta^2}\int_{\V_\delta}\bm{u}(\bx)\cdot\left(\int_{\Omega_\delta} R_\delta(\bx,\by)(\be_\delta(\bx)-\be_\delta(\by))\mathd \by\right)\mathd \bx\right|\nonumber\\
\le&\left[\left(\frac{1}{\delta^2}\int_{\Omega_\delta}\left(\int_{\V_\delta} R_\delta(\bx,\by)|\bar{\bm{\psi}}(\bx)-\bar{\bm{\psi}}(\by)|^2\mathd \by\right)\mathd \bx\right)^{1/2}
+\left(\frac{1}{\delta^2}\int_{\V_\delta}|\bm{u}(\bx)|^2\left(\int_{\Omega_\delta} R_\delta(\bx,\by)\mathd \by\right)\mathd \bx\right)^{1/2}\right]\nonumber\\
&\left(\frac{1}{\delta^2}\int_{\Omega_\delta}\left(\int_{\V_\delta} R_\delta(\bx,\by)|\be_\delta(\bx)-\be_\delta(\by)|^2\mathd \by\right)\mathd \bx\right)^{1/2}\nonumber\\
\le&C\left(
\|\bm{\psi}\|_{H^1(\Omega)}+\sqrt{\delta}\|\bff\|_{H^1(\Omega)}
\right)\left(\frac{1}{\delta^2}\int_{\Omega_\delta}\left(\int_{\V_\delta} R_\delta(\bx,\by)|\be_\delta(\bx)-\be_\delta(\by)|^2\mathd \by\right)\mathd \bx\right)^{1/2}\nonumber\\
&+\frac{C}{\delta^2}\int_{\Omega_\delta}\left(\int_{\V_\delta} R_\delta(\bx,\by)|\be_\delta(\bx)-\be_\delta(\by)|^2\mathd \by\right)\mathd \bx.\nonumber
\end{align}
Moreover,
\begin{align}
  \label{eq:p-error-7}
&  \frac{1}{\delta^2}\int_{\Omega_\delta}\left(\int_{\V_\delta} R_\delta(\bx,\by)|\be_\delta(\bx)-\be_\delta(\by)|^2\mathd \by\right)\mathd \bx\\
\le&\frac{2}{\delta^2}\int_{\Omega_\delta}\left(\int_{\V_\delta} R_\delta(\bx,\by)|\be_\delta(\bx)|^2\mathd \by\right)\mathd \bx+
\frac{2}{\delta^2}\int_{\Omega_\delta}\left(\int_{\V_\delta} R_\delta(\bx,\by)|\be_\delta(\by)|^2\mathd \by\right)\mathd \bx\nonumber\\
\le&\frac{2}{\delta^2}\int_{\Omega_\delta}|\be_\delta(\bx)|^2\left(\int_{\V_\delta} R_\delta(\bx,\by)\mathd \by\right)\mathd \bx+
\frac{2}{\delta^2}\int_{\Omega_\delta}\left(\int_{\V_\delta} R_\delta(\bx,\by)|\bfu(\by)|^2\mathd \by\right)\mathd \bx\nonumber\\
\le& Q^2+C\delta\|\bm{f}\|_{H^1(\Omega)}^2.\nonumber
\end{align}
Combining \eqref{eq:p-error-6} and \eqref{eq:p-error-7}, we get
\begin{align}
  \label{eq:p-error-8}
  &  \left|\frac{1}{\delta^2}\int_{\Omega_\delta}\bar{\bm{\psi}}(\bx)\cdot\left(\int_{\V_\delta} R_\delta(\bx,\by)(\be_\delta(\bx)-\be_\delta(\by))\mathd \by\right)\mathd \bx\right|\nonumber\\
\le&\left(
\|d_\delta  \|_{L^2(\Omega_\delta)}+\sqrt{\delta}\|\bff\|_{H^1(\Omega)}\right)(Q+\sqrt{\delta}\|\bm{f}\|_{H^1(\Omega)})+Q^2+\delta\|\bm{f}\|_{H^1(\Omega)}^2
\end{align}
Substituting \eqref{eq:p-error-5} and \eqref{eq:p-error-8} in \eqref{eq:p-error-4},
\begin{align}
  \label{eq:p-error-9}
  \|d_\delta\|_{L^2(\Omega_\delta)}^2\le& Q^2+C\left(\|d_\delta\|_{L^2(\Omega_\delta)}+\sqrt{\delta}\|\bff\|_{H^1(\Omega)}\right)\left(Q+\sqrt{\delta}\|\bm{f}\|_{H^1(\Omega)}\right)
+\|\bar{\bm{\psi}}\|_{L^2(\Omega_\delta)}\|\bm{r}_\bfu\|_{L^2(\Omega)}\nonumber\\
&+C\sqrt{\delta}\|\bff\|_{H^1(\Omega)}\|d_\delta\|_{L^2(\Omega)}\nonumber\\
\le& Q^2+C\left(\|d_\delta\|_{L^2(\Omega_\delta)}+\sqrt{\delta}\|\bff\|_{H^1(\Omega)}\right)\left(Q+\sqrt{\delta}\|\bm{f}\|_{H^1(\Omega)}\right)\nonumber\\
&+\left(\|d_\delta\|_{L^2(\Omega_\delta)}+\|\be_\delta\|_{L^2(\Omega_\delta)}\right)(\|\bm{r}_\bfu\|_{L^2(\Omega_\delta)})+C\sqrt{\delta}\|\bff\|_{H^1(\Omega)}\|d_\delta\|_{L^2(\Omega)}
\end{align}
On the other hand, using Lemma  \ref{cor:l2-inner-ave}, 
{ we have}
%$\|d_\delta\|_{L^2(\Omega)}$ can be bounded by $\|d_\delta\|_{L^2(\Omega_\delta)}$.
\begin{align}
\label{eq:p-inner-global}
  \|d_\delta\|_{L^2(\Omega)}^2\le C\int_\Omega\int_{\Omega}|d_\delta(\bx)-d_\delta(\by)|^2 R_{\delta}(\bx,\by)\mathd\by\mathd\bx+
C\|d_\delta\|_{L^2(\Omega_\delta)}^2.
\end{align}
Then it follows from \eqref{eq:u-error} and above inequality
\begin{align}
  \label{eq:p-error}
  \|d_\delta\|_{L^2(\Omega)}^2
\le& Q^2+C\left(\|d_\delta\|_{L^2(\Omega)}+\sqrt{\delta}\|\bff\|_{H^1(\Omega)}\right)\left(Q+\sqrt{\delta}\|\bm{f}\|_{H^1(\Omega)}\right)\nonumber\\
&+\left(\|d_\delta\|_{L^2(\Omega)}+\|\be_\delta\|_{L^2(\Omega_\delta)}\right)(\|\bm{r}_\bfu\|_{L^2(\Omega_\delta)})
\end{align}
Theorem \ref{thm:integral_error} gives that
\begin{align}
  \label{eq:est-error}
  \|\bm{r}_\bfu\|_{L^2(\Omega)}\le C\delta\|\bff\|_{H^1(\Omega)},\quad \|r_p\|_{L^2(\Omega)}\le C\delta\|\bff\|_{H^1(\Omega)}
\end{align}
Following Lemma \ref{cor:l2-inner} and \eqref{eq:u-error}, we have
\begin{align*}
%  \label{eq:u-error-est-0}
  \|\bm{e}_\delta\|_{L^2(\Omega_\delta)}^2\le Q^2\le C\sqrt{\delta}\|\bff\|_{H^1(\Omega)}\|\bm{e}_\delta\|_{L^2(\Omega_\delta)}+C\delta\|\bff\|_{H^1(\Omega)}^2+C\sqrt{\delta}\|\bff\|_{H^1(\Omega)}\|d_\delta\|_{L^2(\Omega)}
\end{align*}
which implies that
\begin{align}
  \label{eq:u-error-est-1}
  \|\bm{e}_\delta\|_{L^2(\Omega_\delta)}^2\le C\delta\|\bff\|_{H^1(\Omega)}^2+C\sqrt{\delta}\|\bff\|_{H^1(\Omega)}\|d_\delta\|_{L^2(\Omega)}
\end{align}
Consequently, $Q^2$ is bounded by
\begin{align}
  \label{eq:est-Q}
  Q^2\le  C\delta\|\bff\|_{H^1(\Omega)}^2+C\sqrt{\delta}\|\bff\|_{H^1(\Omega)}\|d_\delta\|_{L^2(\Omega)}
\end{align}
Now, we have the bound of $\|d_\delta\|_{L^2(\Omega)}$ from \eqref{eq:p-error} and \eqref{eq:est-Q},
\begin{align*}
%  \label{eq:p-error-est-0}
   \|d_\delta\|_{L^2(\Omega)}^2\le& C\delta\|\bff\|_{H^1(\Omega)}^2+C\sqrt{\delta}\|\bff\|_{H^1(\Omega)}\|d_\delta\|_{L^2(\Omega)}+\left(\|d_\delta\|_{L^2(\Omega)}+\sqrt{\delta}\|\bff\|_{H^1(\Omega)}\right)Q
\nonumber\\
\le& C\delta\|\bff\|_{H^1(\Omega)}^2+C\sqrt{\delta}\|\bff\|_{H^1(\Omega)}\|d_\delta\|_{L^2(\Omega)}+\left(\frac{1}{2}\|d_\delta\|_{L^2(\Omega)}^2+\delta\|\bff\|_{H^1(\Omega)}^2\right) \nonumber
\end{align*}
Therefore
\begin{align}
  \label{eq:p-error-est-1}
   \|d_\delta\|_{L^2(\Omega)}^2\le& C\delta\|\bff\|_{H^1(\Omega)}^2+C\sqrt{\delta}\|\bff\|_{H^1(\Omega)}\|d_\delta\|_{L^2(\Omega)}
\end{align}
Then the bound of $\|d_\delta\|_{L^2(\Omega)}$ is obtained
\begin{align}
  \label{eq:p-error-est}
   \|d_\delta\|_{L^2(\Omega)}\le& C\sqrt{\delta}\|\bff\|_{H^1(\Omega)}.
\end{align}
The bound of $\|\bm{e}_\delta\|_{L^2(\Omega_\delta)}$ follows from \eqref{eq:u-error-est-1} and \eqref{eq:p-error-est},
\begin{align}
  \label{eq:u-error-est}
  \|\bm{e}_\delta\|_{L^2(\Omega_\delta)}\le C\sqrt{\delta}\|\bff\|_{H^1(\Omega)}.
\end{align}
and 
\begin{align}
  \label{eq:p-error-final}
  \|p-p_\delta\|_{L^2(\Omega)}\le \|d_\delta\|_{L^2(\Omega)}+|\bar{d}_\delta|\le C\sqrt{\delta}\|\bff\|_{H^1(\Omega)},
\end{align}
where $\bar{d}_\delta=\frac{1}{|\Omega|}\int_\Omega d_\delta(\bx)\mathd\bx$ and we use the fact that
\begin{align*}
  |\bar{d}_\delta|=\frac{1}{|\Omega|}\left|\int_\Omega d_\delta(\bx)\mathd\bx\right|\le \frac{1}{\sqrt{|\Omega|}}\|d_\delta\|_{L^2(\Omega)}.
\end{align*}

Finally, the bound of $\|\bm{e}_\delta\|_{H^1(\Omega_\delta)}$ can be derived from % \eqref{eq:u-error}, Corollary \ref{cor:coercivity-inner}, Theorem \ref{thm:integral_error} and
\begin{align}
  \label{eq:u-error-h1-0}
  \be_\delta(\bx)=&\frac{1}{w_\delta(\bx)}\int_{\Omega} R_\delta(\bx,\by)\be_\delta(\by)\mathd \by+\frac{1}{2w_\delta(\bx)}\int_\Omega R_\delta(\bx,\by)(\bx-\by)d_\delta(\by)\mathd\by
-\delta^2\bm{r}_\bfu(\bx).
\end{align}
We are left with estimating the three terms on the right hand side one by one. The third term is easy to bound using Theorem \ref{thm:integral_error},
\begin{align*}
%  \label{eq:u-error-h1-1}
  \|\delta^2\nabla\bm{r}_\bfu(\bx)\|_{L^2(\Omega_\delta)}\le \delta^2\|\bff\|_{H^1(\Omega)}
\end{align*}

Notice that for any $\bx\in \Omega_\delta$, $w_\delta(\bx)$ is a positive constant. Then we have
\begin{align*}
 % \label{eq:u-error-h1-2}
  &\|\nabla\left(\frac{1}{2w_\delta(\bx)}\int_\Omega R_\delta(\bx,\by)(\bx-\by)d_\delta(\by)\mathd\by\right)\|_{L^2(\Omega_\delta)}^2\\
\le& C\int_{\Omega_\delta}\left|\int_{\Omega}\nabla_{\bx} R_\delta(\bx,\by)(\bx-\by)d_\delta(\by)\mathd\by\right|^2\mathd\bx+
C\int_{\Omega_\delta}\left(\int_{\Omega} R_\delta(\bx,\by)d_\delta(\by)\mathd\by\right)^2\mathd\bx\nonumber\\
\le& \frac{C}{\delta^2}\int_{\Omega}\left|\int_{\Omega}{|R'_\delta(\bx,\by)|}|\bx-\by|^2 d_\delta(\by)\mathd\by\right|^2\mathd\bx+
C\int_{\Omega}\left(\int_{\Omega} R_\delta(\bx,\by)d_\delta(\by)\mathd\by\right)^2\mathd\bx\nonumber\\
\le& C\int_{\Omega}\left(\int_{\Omega}{|R'_\delta(\bx,\by)|} d_\delta(\by)\mathd\by\right)^2\mathd\bx+C\int_{\Omega}\left(\int_{\Omega} R_\delta(\bx,\by)d_\delta(\by)\mathd\by\right)^2\mathd\bx
\nonumber\\
\le& C\|d_\delta\|_{L^2(\Omega)}^2\le C\sqrt{\delta}\|\bff\|_{H^1(\Omega)}.\nonumber
\end{align*}
where%\footnote{\color{red} is the function before $R$ or $R'$?}
$${R'_\delta(\bx,\by)}=C_\delta R'\left(\frac{|\bx-\by|^2}{4\delta^2}\right),\quad R'(r)=\frac{\mathd }{\mathd r}R(r).$$

% Direct calculation shows that
% \begin{align}
%   \label{eq:u-error-h1-2}
%   \|\nabla\left(\frac{1}{2w_\delta(\bx)}\int_\Omega R_\delta(\bx,\by)(\bx-\by)d_\delta(\by)\mathd\by\right)\|_{L^2(\Omega_\delta)}\le C\|d_\delta\|_{L^2(\Omega)}\le C\sqrt{\delta}\|\bff\|_{H^1(\Omega)}
% \end{align}
{The first term of \eqref{eq:u-error-h1-0}
can be split into two terms}
\begin{align*}
%  \label{eq:u-error-h1-3}
\frac{1}{w_\delta(\bx)}\int_{\Omega} R_\delta(\bx,\by)\be_\delta(\by)\mathd \by=\frac{1}{w_\delta(\bx)}\int_{\Omega_\delta} R_\delta(\bx,\by)\be_\delta(\by)\mathd \by
+\frac{1}{w_\delta(\bx)}\int_{\V_\delta} R_\delta(\bx,\by)\bfu(\by)\mathd \by
\end{align*}
Using Lemma \ref{lem:u-boundary},
\begin{align*}
%  \label{eq:u-error-h1-4}
  \|\nabla\left(\frac{1}{w_\delta(\bx)}\int_{\V_\delta} R_\delta(\bx,\by)\bfu(\by)\mathd \by\right)\|_{L^2(\Omega_\delta)}\le C\sqrt{\delta}\|\bff\|_{H^1(\Omega)}
\end{align*}
And it follows from Lemma \ref{cor:coercivity-inner} and \eqref{eq:u-error},
\begin{align*}
%  \label{eq:u-error-h1-5}
  \|\nabla\left(\frac{1}{w_\delta(\bx)}\int_{\Omega_\delta} R_\delta(\bx,\by)\be_\delta(\by)\mathd \by\right)\|_{L^2(\Omega_\delta)}\le C\sqrt{\delta}\|\bff\|_{H^1(\Omega)}.
\end{align*}
Hence, the proof is completed.
\end{proof}

% With some minor modifications of the proof of above theorem, we could also get following result which will play important role in the analysis of the
% discrete scheme.
% \begin{theorem}
%   \label{thm:stability}
% Denote
% \begin{align}
%   \label{eq:Lt}
% \bm{\mathcal{L}}_\delta(\bfu,p)=&-\frac{1}{\delta^2}\int_{\Omega_\delta} R_\delta(\bx,\by)(\bfu(\bx)-\bfu(\by))\mathd \by
% +\D \frac{1}{2\delta^2}\int_\Omega R_\delta(\bx,\by)(\bx-\by) p(\by)\mathd \by,\\
% &-\frac{\bfu(\bx)}{\delta^2}\int_{\V_\delta} R_\delta(\bx,\by)\mathd \by,\nonumber\\
% {\mathcal{K}}_\delta(\bfu,p)=& \frac{1}{2\delta^2}\int_{\Omega_\delta} R_\delta(\bx,\by)(\bx-\by)\cdot \bfu(\by)\mathd \by - \int_\Omega \bar{R}_\delta(\bx,\by)(p(\bx)-p(\by))\mathd \by
% \end{align}
% and let $\bfu(\bx),p(\bx)$ satisfying the nonlocal Stokes system with $\bm{r}_{\bm{u}}\in H^1(\M_\delta),\; r_p\in H^1(\M)$,
%   \begin{align*}
%     \bm{\mathcal{L}}_\delta(\bfu,p)(\bx)=&\bm{r}_{\bm{u}}(\bx),\quad \bx\in \M_\delta,\\
% \mathcal{K}_\delta(\bfu,p)(\bx)=& r_{p}(\bx),\quad \bx\in \M
%   \end{align*}
% then, there exists $C>0$ such that
%   \begin{align*}
%     \|\bfu\|_{H^1(\M_\delta)}+\|p\|_{L^2(\M)}\le C (\|\bm{r}_{\bfu}\|_{L^2(\M_\delta)}+\|r_p\|_{L^2(\M)})+C t(\|\nabla \bm{r}_{\bfu}\|_{L^2(\M_\delta)}+\|\nabla r_p\|_{L^2(\M)}).
%   \end{align*}
% \end{theorem}

%\input{discretization}

\section{Discussion and Conclusion}

In this paper, we propose a nonlocal model for linear steady Stokes equation with no-slip boundary condition. The main idea is to use volume constraint to enforce 
the no-slip boundary condition and add a relaxation term in the divergence free condition to maintain the well-posedness of the nonlocal system.  As the nonlocal
horizon paramter $\delta$ approaches 0, the solution of the nonlocal system converges to 
the solution of the original Stoke equation, assuming that the solution to the latter is sufficiently smooth.

In terms of future work, one may examine the convergence with minimal regularity assumptions on the local systems. It is also interesting to consider the numerical discretizations.
From the nonlocal system, we can derive a numerical scheme for the original Stokes system on point cloud. 
Assume we are given
a set of sample points $P$ sampling the domain $\M$ and a subset $S\subset P$ sampling the boundary of $\M$.
In addition, assume we are given
one vector $\bV = (V_1, \cdots, V_n)^t$ where $V_i$ is an volume weight of $\bx_i$ in $\M$, so that for any $C^1$
function $f$ on $\M$, $\int_\M f(\bx) \mathd\bx$ can be approximated by
$\sum_{\bx_i\in\Omega} f(\bx_i) V_i$.

Then, the nonlocal Stokes system \eqref{eq:stokes-nonlocal} can be discretized as following.
\begin{align*}
  -\frac{1}{\delta^2}\sum_{\bx_j\in \Omega} R_\delta(\bx_i,\bx_j)(\bfu_i-\bfu_j)V_j+\frac{1}{2\delta^2}\sum_{\bx_j\in \Omega} R_\delta(\bx_i,\bx_j)(\bx_i-\bx_j)p_jV_j
&\\
&\hspace{-3.3cm} = \sum_{\bx_j\in \Omega} \bar{R}_\delta(\bx_i,\bx_j)\bff_jV_j,\quad \bx_i\in \Omega_\delta\\
\frac{1}{2\delta^2}\sum_{\bx_j\in \Omega} R_\delta(\bx_i,\bx_j)(\bx_i-\bx_j)\bfu_jV_j-\sum_{\bx_j\in \Omega} \bar{R}_\delta(\bx_i,\bx_j)(p_i-p_j)V_j=& 0,\quad \bx_i\in \Omega,\\
\bfu_i=& 0,\quad \bx_i\in \V_\delta.
\end{align*}
This scheme is very simple and easy to implement. However, the accuracy is relatively low. We can show that the error of above scheme is $O\left(\frac{h}{\delta^2}+\delta\right)$,
where $h$ is the average distance among the sample points in $P$. The first term $h/\delta^2$ comes from the error of the numerical integral and the second term $\delta$ is from error between nonlocal system and the original Stoke equation. 
%One way to improve the accuracy is to use high order quadrature rule in the computing of the integral transforms. 
%Formally, the error is $O\left(\frac{h^k}{\delta^2}+\delta\right)$, where $k$ is the order of the accuracy of the quadrature rule. To get the optimal error estimate, we should take
%$\delta=O(h^{k/3})$. If $k$ is large enough $(k>3)$, the nonlocal scale $\delta$ can be even smaller than $h$ and the numerical scheme is still convergent.
Further improvement and studies of asymptotically compatible scheme \cite{TD14} are interesting questions to be explored further.
%was proposed for nonlocal system. One scheme is asymptotically compatible means that 
% it is not only a convergent numerical method for the nonlocal problem with
% fixed $\delta$, but also preserves the convergence as $\delta$ goes to zero. Obviously, above scheme is not asymptotically compatible.

\appendix
%\setcounter{section}{1}
% \vspace{10mm}
%\newpage

\section{Formal derivation of the nonlocal Stokes model}
\label{sec:appendix1}

Based on Assumptions \ref{assumption} on the nonlocal kernels, we give some formal derivation of the nonlocal Stokes model from its local counterpart.

First, {for $\bx\in \Omega_\delta$,  we multiply $\bar{R}_\delta(\bx,\by)$ on both sides of the first equation of the Stokes system \eqref{eq:stokes} evaluated at $\by\in \Omega$ and taking integral  with respect to $\by$} over $\Omega$,
\begin{align*}
   \int_\Omega \bar{R}_\delta(\bx,\by) \Delta \bfu(\by)\mathd \by -\int_\Omega \bar{R}_\delta(\bx,\by) \nabla p(\by)\mathd \by = \int_\Omega \bar{R}_\delta(\bx,\by)\bff(\by)\mathd \by,\quad \bx\in \Omega_\delta
\end{align*}
For the left hand side, we apply integration by parts and using the  { property $\bar{R}_\delta(\bx,\by)=0$ for  $\by\in \partial \Omega$ and the} relation between $\bar{R}$ and $R$, 

\begin{equation}
\label{eq:deriv2}
\begin{aligned}
&    \frac{1}{2\delta^2}\int_\Omega R_\delta(\bx,\by) ({\by-\bx})\cdot\nabla \bfu(\by)\mathd \by - \frac{1}{2\delta^2}\int_\Omega R_\delta(\bx,\by) ({\by-\bx}) p(\by)\mathd \by\\
&\quad = \int_\Omega \bar{R}_\delta(\bx,\by)\bff(\by)\mathd \by,\quad \bx\in \Omega_\delta
\end{aligned}
\end{equation}

For the first term of the left hand side, the derivation in \cite{Shi-vc} proceeds with an approximation by Taylor expansion
%\footnote{{\color{red} Note that in principle, by Dirichlet condition on $\bfu$ we can do another integration by part to get rid of the gradient on $\bfu$, e.g., 
%\begin{align*}
%&-\frac{1}{2\delta^2}\int_\Omega R_\delta(\bx,\by) (\bx-\by)\cdot\nabla \bfu(\by)\mathd \by\\
%&=-\frac{1}{2\delta^2}\int_\Omega R_\delta(\bx,\by)\nabla \cdot [(\bx-\by)\bfu(\by)]\mathd \by
%-\frac{n}{2\delta^2}\int_\Omega R_\delta(\bx,\by) \bfu(\by)\mathd \by\\
%&=- \frac{1}{4\delta^4}\int_\Omega R^\prime_\delta(\bx,\by) |\bx-\by|^2\bfu(\by)\mathd \by
%-\frac{n}{2\delta^2}\int_\Omega R_\delta(\bx,\by) \bfu(\by)\mathd \by
%\end{align*}
%so the question is, if the objective is to get an integral form avoiding derivatives on $\bfu$, then what's the rationale to do Taylor on this term instead of another integration by part.} {\color{blue} 
%\begin{align*}
%&-\frac{1}{2\delta^2}\int_\Omega R_\delta(\bx,\by) (\bx-\by)\cdot\nabla \bfu(\by)\mathd \by\\
%&=\frac{1}{2\delta^2}\int_\Omega R_\delta(\bx,\by) (\bx-\by)\cdot\nabla_{\by} (\bfu(\bx)-\bfu(\by))\mathd \by\\
%&=\frac{1}{4\delta^4}\int_\Omega R^\prime_\delta(\bx,\by) |\bx-\by|^2(\bfu(\bx)-\bfu(\by))\mathd \by
%+\frac{n}{2\delta^2}\int_\Omega R_\delta(\bx,\by) (\bfu(\bx)-\bfu(\by))\mathd \by\\
%&=\int_{\Omega} \Delta_{\by} \bar{R}_\delta(\bx,\by) (\bfu(\bx)-\bfu(\by))\mathd \by
%\end{align*}
%So the kernel becomes $\Delta_{\by} \bar{R}_\delta(\bx,\by)$. To get the stability, we need $\Delta_{\by} \bar{R}_\delta(\bx,\by)\ge 0$. We also require that $\mbox{supp}\bar{R}\subset [0,1]$. It seems that these conditions can not be satisfied simultaneously.}} 
for $\bx\in \Omega_\delta$,
\begin{align*}
&\int_\Omega \bar{R}_\delta(\bx,\by) \Delta \bfu(\by)\mathd \by\\
    =&-\frac{1}{2\delta^2}\int_\Omega R_\delta(\bx,\by) (\bx-\by)\cdot\nabla \bfu(\by)\mathd \by\\
    =&-\frac{1}{2\delta^2}\int_\Omega R_\delta(\bx,\by) \left(\bfu(\bx)-\bfu(\by)-\frac{1}{2}\sum_{i,j=1}^n(x_i-y_i)(x_j-y_j)\frac{\partial^2 \bfu(\by)}{\partial y_i\partial y_j}\right)\mathd \by+O(\delta)\\
    =&-\frac{1}{2\delta^2}\int_\Omega R_\delta(\bx,\by) \left(\bfu(\bx)-\bfu(\by)\right)\mathd \by+\frac{1}{2}\sum_{i,j=1}^n\int_\Omega \frac{\partial}{\partial y_i}\bar{R}_\delta(\bx,\by) (x_j-y_j)\frac{\partial^2 \bfu(\by)}{\partial y_i\partial y_j}\mathd \by+O(\delta)\\
    =&-\frac{1}{2\delta^2}\int_\Omega R_\delta(\bx,\by) \left(\bfu(\bx)-\bfu(\by)\right)\mathd \by+\frac{1}{2}\int_\Omega \bar{R}_\delta(\bx,\by) \Delta \bfu(\by)\mathd \by+O(\delta)\\
    =&-\frac{1}{\delta^2}\int_\Omega R_\delta(\bx,\by) \left(\bfu(\bx)-\bfu(\by)\right)\mathd \by+O(\delta)\,.
\end{align*}
By dropping $O(\delta)$ term, we obtain
{
\begin{align*}
  &  -\frac{1}{\delta^2}\int_\Omega R_\delta(\bx,\by) (\bfu(\bx)-\bfu(\by))\mathd \by +\frac{1}{2\delta^2}\int_\Omega R_\delta(\bx,\by) (\bx-\by) p(\by)\mathd \by\\
  &\qquad = \int_\Omega \bar{R}_\delta(\bx,\by)\bff(\by)\mathd \by,\quad \bx\in \Omega_\delta\,.
\end{align*}
}
From the derivation, it would appear that the error in the approximation of the left hand side is formally of order  $O(\delta)$.
%For technical reasons in the analysis, we modify the integral domain on the right hand side from $\Omega$ to $\Omega_\delta$. This 
%approximation formally also gives $O(\delta)$ error.

The derivation of the second equation of the nonlocal model \eqref{eq:stokes-nonlocal-intro} is much easier. We also multiply $\bar{R}_\delta(\bx,\by)$ in the divergence free equation and carry out integration by parts over $\Omega$
\begin{align*}
    \int_\Omega R_\delta(\bx,\by)(\bx-\by)\cdot \bfu(\by)\mathd \by=0\,.
\end{align*}
Then a stablization term
%, $-\int_\Omega \bar{R}_\delta(\bx,\by)(p(\bx)-p(\by))\mathd \by$,
{ that mimics a nonlocal analog of the multiple of $\delta^2 \Delta p$
is added to the above to obtain the second equation of the nonlocal model \eqref{eq:stokes-nonlocal-intro}:
%\begin{align*}
%  \int_\Omega R_\delta(\bx,\by)(\bx-\by)\cdot \bfu(\by)\mathd \by-\int_\Omega \bar{R}_\delta(\bx,\by)(p(\bx)-p(\by))\mathd \by=0.
%\end{align*}
We remark that the stablization term is $O(\delta^2)$ so that its presence does not affect the order of the overall approximation.}

\section{Some basic estimates on the local Stokes system}
\label{sec:appendix3}
\begin{lemma}
  \label{lem:u-boundary}
Let $\bfu(\bx)$ be the solution of the Stokes system \eqref{eq:stokes} and $\bff\in H^1(\M)$, then there are generic constants $C>0$ and $T_0>0$, depending only on $\M$ and $\p\M$, such that for any $\delta<T_0$, 
%\footnote{\color{red} should be $\|f\|_{L^2}$? Note that many places that used the $H^1$ norm need to be changed if this is the case}
\begin{eqnarray}
  %\label{eq:u-boundary}
  \int_{\V_\delta}|\bfu(\by)|^2\mathd\by\le C\delta^{3}\|\bff\|_{L^2(\M)}^2. \nonumber
\end{eqnarray}
\end{lemma}
\begin{proof} Since $\p \M$ is compact and $C^{\infty}$ smooth. Consequently, it is well known that $\p \M$ has positive reaches \cite{dey2011approximating}.%\footnote{\textcolor{red}{Could use a reference on this terminology "positive reaches"}}
, which means that
there exists $T_0>0$ only depends on $\p\M$, if $t<T_0$,
$\V_\delta$ can be parametrized as $(\mathbf{z}(\by),\tau)\in \p\M\times [0,1]$, where $\by=\mathbf{z}(\by)+\tau(\mathbf{z}'(\by)-\mathbf{z}(\by))$
and $\left|\det\left(\frac{\mathd \by}{\mathd (\mathbf{z}(\by),\tau)}\right)\right|\le C\delta$ and $C>0$ is a constant only depends on $\M$ and $\p\M$.
Here $\mathbf{z}'(\by)$ is the intersection point between $\p\M'$ and the line determined by $\mathbf{z}(\by)$ and $\by$. The parametrization is illustrated in
Fig.\ref{fig:boundary}.
  \begin{figure}
    \centering
    \includegraphics[width=0.6\textwidth]{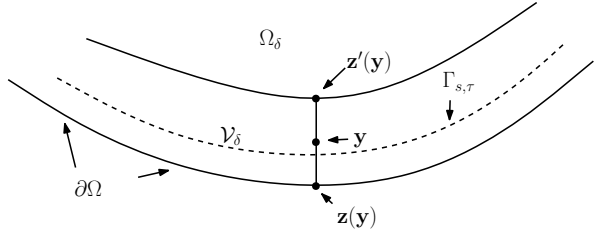}
    \caption{Parametrization of $\mathcal{V}_\delta$.}
    \label{fig:boundary}
  \end{figure}

First, we have
  \begin{eqnarray}
%  \label{eq:u-1}
&&  \int_{\V_\delta}|\bfu(\by)|^2\mathd\by=\int_{\V_\delta}|\bfu(\by)-\bfu(\mathbf{z}(\by))|^2\mathd\by\nonumber\\
&=&\int_{\V_\delta}\left|\int_{0}^1\frac{\mathd }{\mathd s}\bfu(\by+s(\mathbf{z}(\by)-\by))\mathd s\right|^2\mathd\by\nonumber\\
&=&\int_{\V_\delta}\left|\int_{0}^1(\mathbf{z}(\by)-\by)\cdot \nabla \bfu(\by+s(\mathbf{z}(\by)-\by))\mathd s\right|^2\mathd\by\nonumber\\
&\le&C\delta^2\int_{\V_\delta}\int_{0}^1\left| \nabla \bfu(\by+s(\mathbf{z}(\by)-\by))\right|^2\mathd s\mathd\by\nonumber\\
&\le & C\delta^2\sup_{0\le s\le 1}\int_{\V_\delta}\left| \nabla \bfu(\by+s(\mathbf{z}(\by)-\by))\right|^2\mathd\by.\nonumber
\end{eqnarray}
Here, we use the fact that $\|\mathbf{z}(\by)-\by\|_2\le 2\delta$ to get the second last inequality.

Then, the proof can be completed by following estimation.
\begin{eqnarray}
 % \label{eq:u-2}
&&  \int_{\V_\delta}\left| \nabla \bfu(\by+s(\mathbf{z}(\by)-\by))\right|^2\mathd\by\nonumber\\
&\le& C\delta\int_0^1\int_{\p\M}\left| \nabla \bfu(\mathbf{z}(\by)+(1-s)\tau(\mathbf{z}'(\by)-\mathbf{z}(\by)))\right|^2\mathd\mathbf{z}(\by)\mathd \tau\nonumber\\
&\le& C\delta\sup_{0\le \tau\le 1}\int_{\p\M}\left| \nabla \bfu(\mathbf{z}+(1-s)\tau(\mathbf{z}'-\mathbf{z}))\right|^2\mathd\mathbf{z}\nonumber\\
&\le& C\delta\sup_{0\le \tau\le 1}\int_{\Gamma_{s,\tau}}\left| \nabla \bfu(\tilde{\mathbf{z}})\right|^2\mathd\tilde{\mathbf{z}}\nonumber\\
&\le &C\delta \|\bfu\|_{H^2(\M)}^2\le C\delta \|\bff\|_{L^2(\M)}^2,\nonumber
\end{eqnarray}
where $\Gamma_{s,\tau}$ is a $k-1$ dimensional manifold given by $\Gamma_{s,\tau}=\left\{\mathbf{z}+(1-s)\tau(\mathbf{z}'-\mathbf{z}):\mathbf{z}\in \p\M\right\}$.
We use the trace theorem to get the second last inequality and the last inequality is due to that $\bfu$ is the solution of the Stokes system \eqref{eq:stokes}
%\footnote{{\color{red} from the last step, it would seems like we should have
%$\|\bfu\|_{H^2(\M)}\le C \|\bff\|_{L^2(\M)}.$}
%{\color{black}Yes, we can bound it by $\|\bff\|_{L^2(\M)}$, but here we only need $\|\bff\|_{H^1(\M)}$. }}
\end{proof}

\section{Divergence estimation \eqref{eq:bound-v} \eqref{eq:bound-psi}}
\label{sec:appendix4}
\begin{theorem}(Theorem III.3.1 in \cite{NS-book})
\label{thm:div}
  Let $\Omega$ be a bounded domain of $\mathbb{R}^n,\; n\ge 2$, such that
  $$\Omega=\bigcup_{k=1}^N \Omega_k,\quad N\ge 1,$$
  where each $\Omega_k$ is star-shaped with respect to some open ball $B_k$ with $\bar{B}_k\subset \Omega_k$. Then, given $f\in L^q(\Omega),\; 1<q<\infty$, satisfying $\int_\Omega f(\bx)\mathd \bx=0$, there exists at least one solution $\bm{v}\in W^{1,q}_0(\Omega)$ to 
  $$\nabla \cdot \bm{v}(\bx)=f(\bx),\quad \bx\in\Omega,$$
  and 
  $$\|\bm{v}\|_{1,q}\le c\|f\|_q.$$
  Furthermore, the constant $c$ admits the following estimate:
  $$c\le c_0C\left(\frac{d(\Omega)}{R_0}\right)^n\left(1+\frac{d(\Omega)}{R_0}\right),$$
  where $R_0$ is the smallest radius of the balls $B_k$, $d(\Omega)$ is the diameter of $\Omega$, $c_0=c_0(n,q)$
  and $C$ is an upper bound for the constants $C_k$ given as following,
  \begin{align*}
      C_1&=1+\frac{|\Omega_1|^{1-1/q}}{|F_1|^{1-1/q}}\\
      C_k&=\left(1+\frac{|\Omega_k|^{1-1/q}}{|F_k|^{1-1/q}}\right)\prod_{i=1}^{k-1}(1+|F_i|^{-(1-1/q)}|D_i-\Omega_i|^{1-1/q}),\quad k\ge 2,
  \end{align*}
  $F_i=\Omega_i\cap D_i$ and $D_i=\bigcup_{s=i+1}^N \Omega_s$.\footnote{Since $\Omega$ is connected, we can always label sets $F_i$ in such a way that $|F_i|\neq 0,\; i=1,\cdots,N$.}
\end{theorem}

Based on above theorem, to get the constant independent on $\delta$ in \eqref{eq:bound-psi}, we need to find decomposition for $\Omega_\eta,\; 0\le \eta\le \delta_0$ such that corresponding $R_0$ and $|F_i|$ both have uniform lower bound independent on $\eta$ with some $\delta_0>0$. Next, we will give an explicit way to construct the decomposition of $\Omega_\eta$.

 \begin{figure}
    \centering
    \includegraphics[width=0.8\textwidth]{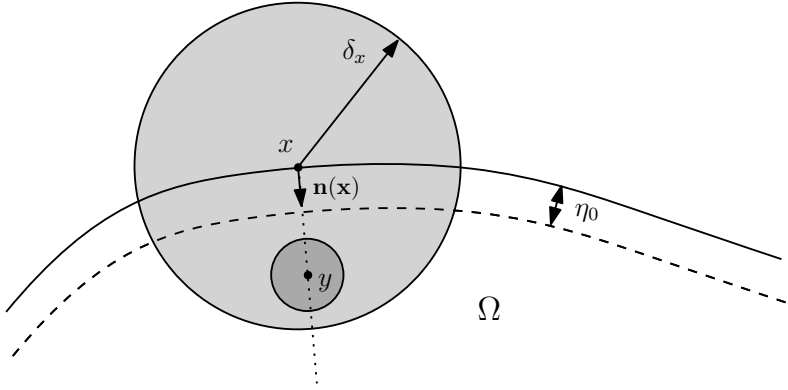}
    \caption{Cover of $\partial \Omega$.}
    \label{fig:div}
  \end{figure}

%We need to show that there exists $c_0>0$ and for any $\eta\in [0,c_0]$, we have
%$$\Omega_\eta=\bigcup_{k=1}^{N_\eta}\Omega_{\eta,k}$$

Under the assumption that the boundary $\partial \Omega$ is $C^2$ smooth, as shown in Fig. \ref{fig:div}, for any point $\bx\in \partial \Omega$, there exists $\delta_{\bx}>0$ such that $$U_{\bx}=\{\bz\in \Omega: |\bz-\bx|<\delta_{\bx}\}$$
is star-shaped with respect to open ball $B(\by,\delta_{\bx}/4)$ with $\by=\bx-\frac{2}{3}\delta_{\bx} \bn(\bx)$, $\bn(\bx)$ is the outer normal of $\partial \Omega$ at $\bx$.

$\displaystyle\bigcup_{\bx\in \partial \Omega}U_{\bx}$ is an open cover of $\partial \Omega$. Since $\partial \Omega$ is compact, there exist $\bx_k\in \partial \Omega,\; k=1,\cdots,N$ such that 
$$\partial \Omega\subset \bigcup_{k=1}^N U_{\bx_k}.$$ Compactness of $\partial \Omega$ also implies that there exists $\displaystyle \eta_0\in \left(0,\frac{1}{2}\min_{1\le k\le N} \delta_{\bx_k}\right)$ such that $$\mathcal{V}_{\eta_0}\subset \bigcup_{k=1}^N U_{\bx_k}.$$ Recall that $\mathcal{V}_{\eta_0}=\{\bx\in \Omega: \mbox{dist}(\bx,\partial \Omega)\le \eta_0\}$.

For any $0\le \eta\le \eta_0/2$, 
$$U_{\bx_k}^\eta=\{\bz\in \Omega_\eta: |\bz-\bx_k|<\delta_{\bx_k}\},\quad k=1,\cdots,N$$ 
are also star-shaped with respect to $B(\by_k,\delta_{\bx_k}/4)$ with $\by_k=\bx_k-\frac{2}{3}\delta_{\bx_k} \bn(\bx_k)$, $\bn(\bx_k)$ is the outer normal of $\partial \Omega$ at $\bx_k$.

On the other hand, compactness of $\bar{\Omega}_{\eta_0}$ gives  $\bz_1,\cdots,\bz_M\in \bar{\Omega}_{\eta_0}$ such that $$\bar{\Omega}_{\eta_0}\subset \bigcup_{k=1}^M B(\bz_k,\eta_0/2).$$

For any $0\le \eta\le \eta_0/2$, 
$$\Omega_\eta=\left(\bigcup_{k=1}^N U_{\bx_k}^\eta\right)\bigcup \left(\bigcup_{k=1}^M B(\bz_k,\eta_0/2)\right) $$
$U_{\bx_k}^\eta$ is star-shaped with respect to $B(\by_k,\eta_0/2)$ and $B(\bz_k,\eta_0/2)$ is star-shaped with respect to itself. It is easy to check based on above decomposition, Theorem \ref{thm:div} implies \eqref{eq:bound-v} and \eqref{eq:bound-psi}.
%\setcounter{section}{2}
% \section{ Proof of Theorem \ref{thm:wellpose}}
% \label{sec:appendix}
 
\bibliographystyle{abbrv}
%\bibliography{poisson}
\bibliography{reference}

\end{document}